\documentclass[11pt,twoside]{amsart}
\usepackage{amsmath,amscd,amsthm,amssymb,amsxtra,latexsym,epsfig,epic,graphics,mathdots,amssymb,mathrsfs,enumitem}
\usepackage[all]{xy}
\usepackage[utf8]{inputenc}
\usepackage{hyperref}
\hypersetup{
	colorlinks=true,
	linkcolor=blue,
	filecolor=magenta, 
	urlcolor=blue,
}
\usepackage{dsfont}
\usepackage{comment}
\usepackage[margin=1.3in]{geometry}

\usepackage{pgf,tikz,pgfplots} 
\usetikzlibrary{arrows,tikzmark,quotes,cd}
\tikzset{
  commutative diagrams/.cd, 
  arrow style=tikz, 
  diagrams={>=stealth}
}\tikzcdset{
	labels={font=\everymath\expandafter{\the\everymath\displaystyle}},
}
\tikzcdset{
  cells={font=\everymath\expandafter{\the\everymath\displaystyle}},
}

\usepackage{dynkin-diagrams}

\usepackage{graphicx,color}

\newtheorem{Theorem}{Theorem}[section]
\newtheorem{MainTheorem}{Theorem}

\newtheorem{Corollary}[Theorem]{Corollary}
\newtheorem{Question}[Theorem]{Question}
\newtheorem{Lemma}[Theorem]{Lemma}
\newtheorem{Proposition}[Theorem]{Proposition}
\theoremstyle{definition}

\newtheorem{Notation}[Theorem]{Notation}
\newtheorem{Example}[Theorem]{Example}
\newtheorem{Remark}[Theorem]{Remark}

\newcommand{\extp}{{\textstyle \bigwedge}} 

\newcommand{\CC}{{\mathbb C}}
\newcommand{\ZZ}{{\mathbb Z}}

\newcommand{\PP}{{\mathbb P}}

\newcommand{\cO}{{\mathcal O}}

\newcommand{\Aut}{{\mathrm{Aut}}}

\DeclareMathOperator{\img}{{im}}
\DeclareMathOperator{\codim}{{codim}}

\DeclareMathOperator{\sing}{Sing}

\DeclareMathOperator{\pic}{Pic}

\DeclareMathOperator{\id}{id}
\DeclareMathOperator{\reddd}{red}
\DeclareMathOperator{\Pic}{Pic}
\DeclareMathOperator{\Fol}{Fol}
\DeclareMathOperator{\IF}{IF}
\DeclareMathOperator{\Dist}{Dist}
\DeclareMathOperator{\GL}{GL}
\DeclareMathOperator{\SL}{SL}

\DeclareMathOperator{\Sp}{Sp}
\DeclareMathOperator{\SO}{SO}
\DeclareMathOperator{\Fl}{Fl}

\newcommand{\fsl}{\mathfrak{sl}}
\newcommand{\U}{\mathcal{U}}
\newcommand{\Q}{\mathcal{Q}}

\newcommand{\cG}{\mathrm{G}}
\newcommand{\cP}{\mathrm{P}}
\newcommand{\cB}{\mathrm{B}}
\newcommand{\cD}{\mathrm{D}}

\newcommand{\OO}{{\mathcal{O}}}
\newcommand{\fol}{{\mathscr{F}}}
\newcommand{\IZ}{{\mathcal I}_Z}

\newcommand{\barra}

\setlist[enumerate]{label=\roman*)} 

\makeatletter 
\def\paragraph{\@startsection{paragraph}{4}%
  \z@\z@{-\fontdimen2\font}%
  {\normalfont\bfseries}}
\makeatother

\begin{document}

\sloppy

\title{Codimension one foliations on homogeneous varieties}

\author[V. Benedetti, D. Faenzi, A. Muniz]{Vladimiro Benedetti, Daniele Faenzi, Alan Muniz }

\address{Institut de Mathématiques de Bourgogne,
UMR CNRS 5584,
Université de Bourgogne et Franche-Comté,
9 Avenue Alain Savary,
BP 47870,
21078 Dijon Cedex,
France}
\email{vladimiro.benedetti@u-bourgogne.fr}
\email{daniele.faenzi@u-bourgogne.fr}
\email{alan.muniz@u-bourgogne.fr}


\begin{abstract}
The aim of this paper is to study codimension one foliations on rational homogeneous spaces, with a focus on the moduli space of foliations of low degree on Grassmannians and cominuscule spaces. Using equivariant techniques, we show that codimension one degree zero foliations on (ordinary, orthogonal, symplectic) Grassmannians of lines, some spinor varieties, some Lagrangian Grassmannians, the Cayley plane (an $E_6$-variety) and the Freudenthal variety (an $E_7$-variety) are identified with restrictions of foliations on the ambient projective space. We also provide some evidence that such results can be extended beyond these cases.
\end{abstract}



\maketitle

\section{Introduction}

Let $X$ be a compact connected complex manifold of dimension $n$. A codimension $p$ holomorphic distribution on $X$ is a rank $n-p$ saturated subsheaf of the tangent bundle $F\subset T_X$; here saturated means that $T_X/F$ is torsion free. If moreover $F$ is stable under the Lie bracket, it is called a foliation. For a fixed line bundle $L \in \Pic(X)$ the space $\Fol_p(X,L)$ of codimension-$p$ foliations with $\det(F) = L^\vee\otimes \omega_X^\vee$ is a locally closed subvariety of $\PP(H^0(\Omega_X^p(L))^\vee)$. Being stable under the Lie bracket translates to a closed condition on the coefficients of the $p$-form, both are known as the Frobenius integrability condition.   

For $p=1$,  as we shall explain in greater detail in \S\ref{section:distributions and foliations}, this condition is conveniently described by the quadratic map:
\[
\psi_X\colon H^0(\Omega_X^1(L)) \to H^0(\Omega^3_X(2L)), \qquad \omega \mapsto \omega \wedge d\omega,
\]
The zero-scheme of $\psi_X$ inside $H^0(\Omega_X^1(L))$ is what we call the locus $\IF(X,L)$ of integrable forms, while we denote by $\Dist(X,L)$ the open set of $1$-forms not vanishing in codimension one, as they correspond to distributions satisfying $\det(F) = L^\vee\otimes \omega_X^\vee$. Then
$$
\Fol(X,L) := \Fol_1(X,L) = \Dist(X,L) \cap \IF(X,L)  \subset \PP(H^0(\Omega_X^1(L))^\vee).
$$

The description of $\Fol_p(X,L)$ for given $X$ and $L$ is an interesting problem in the global theory of holomorphic foliations. 
For $X = \PP^n$ and $p=1$ this problem is already very challenging. In this case we write $L \cong \cO_{\PP^n}(d+2)$, where $d\geq 0$ is traditionally called the \textit{degree} of the foliation. A full description of $\Fol(\PP^n,d+2)$ is only known, at this moment, for $d\leq 2$. For degree $d=0$ every foliation is given by a pencil of hyperplanes so $\Fol(\PP^n,2)$ is isomorphic to the Grassmannian $G(2, n+1)$, the inclusion in
$\PP(H^0(\Omega_{\PP^n}^1(2))^\vee )$ being the Pl\"ucker embedding. It is unknown to us when this fact was first established but we refer to \cite[Chapitre 3]{CD:Lfol} and \cite[Theorem 4.3]{ACM:FanoDist} for proofs, see also Remark \ref{rem:transtorep}. The case $d=1$ was described in 1979 by Jouanolou \cite{Jouanolou}; $\Fol(\PP^n, 3)$ has $2$ irreducible components. The case $d=2$ was established in 1996 by Cerveau and Lins Neto \cite{CLN:components}; $\Fol(\PP^n, 4)$ has $6$ irreducible components. For $d=3$ there exists, until this date, a partial classification due to da Costa, Lizarbe and Pereira \cite{CLP:deg3}; they prove that $\Fol(\PP^n, 5)$ has at least $24$ components, some of them being not generically reduced -- a phenomenon that does not occur in lower degrees.

Many authors studied the geometry of foliations on other manifolds, especially when $X$ is of low dimension (see \cite{brunella:book}) or when $-c_1(F)$ is positive or numerically trivial (see in particular \cite{araujo-druel:fano-foliations, araujo-druel:mukai-foliations} for $-c_1(F)$ ample and \cite{loray-pereira-touzet:trivial} for $c_1(F) \equiv0 $).
However, much less seems to be known about the behaviour of foliations under restriction, our main inspiration being \cite{ACM:FanoDist}, where special attention is paid to the case of complete intersections.

The aim of our work is to describe the space $\Fol(X,d+2)$ of codimension one foliations on a manifold $X$ which is $\cG$-homogeneous for the action of a simple complex Lie group $\cG$, bearing in mind that a prominent role should be played by the representation theory of $\cG$, or of the stabiliser $\cP$ of a point of $X$.
The spaces we consider are Grassmanians in their Pl\"ucker embedding, or more generally  \textit{cominuscule Grassmannians},  see  \S \ref{subsection:cominuscule}, since for these varieties we only need the representation theory of the semisimple part of $\cP$, which affords a major simplification of our analysis. For a $\cG$-homogeneous variety $X$, there
is an irreducible $\cG$-representation $V$ such that $X \subset \PP(V)$ is the minimal $\cG$-equivariant embedding. All line bundles on $X$ are of the form $\cO_X(t)$ for some $t \in \ZZ$, where $\cO_X(1)$ is the $\cG$-linearized hyperplane section bundle of $X \subset \PP(V)$.
Then, considering the natural restriction maps $i_p^*$ of $p$-forms from $\PP(V)$ to $X$, we get the following result. 

\begin{MainTheorem}
Let $X \subset \PP(V)$ be a cominuscule variety.
Then, for $d,p \ge 0$:
\begin{enumerate} \label{mainA}
    \item \label{A-i} The restriction map $i_p^*\colon H^0(\Omega^p_{\PP(V)}(d+2)) \to H^0(\Omega^p_{X}(d+2))$ is surjective;
    \item \label{A-ii} The space of integrable forms $\IF(X,d+2)\subset \PP(H^0(\Omega^1_X(d+2))^\vee)$ is defined by the quadratic equations given by the $\cG$-equivariant inclusion:
    \[   
    H^0(\Omega^3_X(d+2))^\vee \subset S^2 H^0(\Omega^1_X(d+2))^\vee.
    \]
    \end{enumerate}
Assume further that $X$  is a cominuscule  Grassmannian in its minimal embedding. Then:
        \begin{enumerate}[resume] \label{mainA}
    \item \label{A-iii} The map $\pi\colon \Fol(\PP(V),2) \to \Fol(X,2)$ induced by $i_1^*$ is an embedding.
\end{enumerate}
\end{MainTheorem}

The result above is a summary of statements spread along the paper: item \ref{A-i} is Lemma \ref{lem:surjPbComin},
item  \ref{A-ii} is Theorem \ref{thm:folcomin}  and \ref{A-iii} is shown in Corollary \ref{cor:injproj}, see also Poposition \ref{prop:restinject}. 

Then we look in more detail at the case of Grassmannians of lines and a few other cominuscule varieties, where our results are particularly neat.

\begin{MainTheorem} 
Let $X \hookrightarrow \PP(V)$ be the minimal embedding of a cominuscule Grassmannian. Then the restriction of $1$-forms gives an isomorphism:
\[
\Fol(X,2) \cong \Fol(\PP(V),2) \cong G(2,V)
\]
in the following cases:
\begin{enumerate}
    \item 
    Quadrics $Q^n \subset \PP^{n+1}$, with $n\geq 3$;
    \item \label{B-i} Grassmannians of lines, $X=G(2,n)$;
    \item \label{B-iii} spinor varieties $OG(n,2n)$, for $n=4,5$;
    \item \label{B-iv} the Cayley plane, $X=E_6/\cP_1$;
    \item \label{B-ii} the four legendrian varieties appearing in Freudenthal's magic square:
    \[
    LG(3,6), \qquad    G(3,6), \qquad OG(6,12), \qquad E_7/\cP_7.
    \]
\end{enumerate}
\end{MainTheorem}

To our knowledge, the only cases that were known previously are that of quadrics (for this is a consequence of \cite[Theorem 1.3]{ACM:FanoDist}) and of $G(2,5)$ (which is proved in item (5) of \cite[Theorem 1.5]{ACM:FanoDist}). 
A key point here is that in all cases above except \ref{B-ii} the isomorphism takes place because the maps $i_1^*$ and $i_3^*$ are isomorphisms (this is shown for items \ref{B-i}, \ref{B-iii} and \ref{B-iv} respectively in Theorems \ref{thm:isoG2n}, \ref{thm_S_equations} and \ref{thm:E6}), so Theorem \ref{mainA} is enough to conclude.


On the other hand, in case \ref{B-ii} we show in Theorem \ref{thm:fol36} that the map $i_1^*$ induces a linear projection from a distinguished point lying away from $\Fol(\PP(V),2)$ and that this map is actually an isomorphism onto $\Fol(X,2)$. 
We provide a uniform proof for these four cases, based on the observation that the point used to define the projection corresponds to $\cG$-invariant contact form on $\PP(V)$.
This is also the reason for the apparently awkward choice of listing the spinor variety $OG(6,12)$ among the cases of \ref{B-ii} rather than in \ref{B-iii}.

This leads to the expectation that $\Fol(X,d+2) \cong \Fol(\PP(V),d+2)$ for all cominuscule Grassmannians and small $d$, though the evidence we provide is only for $d=0$ and mainly $G(3,n)$. In Corollary \ref{cor:injproj} we prove that $i_1^*$ induces an embedding for primitively embedded cominuscule Grassmannians. In the particular case of the Pl\"ucker embedding of $G(3,n)$, Theorem \ref{thm:embG3n} shows that the ideal of its image agrees with that of $\Fol(G(3,n),2)$ up to degree 2. 
We prove more results in this direction, even slightly more generally than for homogeneous spaces, for instance in Theorem \ref{thm:symplectic_lines} we address the case of isotropic Grassiammians of lines for a skew-symmetric form of maximal rank. 

However, one should be warned that $i_1^*$ sometimes induces a proper inclusion, for instance this happens for the (non-cominuscule) variety of isotropic lines for a non-degenerate quadratic form, see Proposition \ref{prop_OG_equations}, and for products of projective spaces, see Proposition \ref{prop_PP_equations}.

The paper is organized as follows. In Section \ref{section:distributions and foliations} we introduce distributions, foliations and integrable forms taking values on a given line bundle on a manifold, with a focus on projective spaces, whereby defining the quadratic equations of integrability and the restriction maps mentioned above.  In Section \ref{section:homogeneous spaces} we start the discussion of homogeneous spaces, recall the list of cominuscule varieties and provide some basic results about distributions and foliations over them. In Section \ref{sect:grassmanians} we look more closely to Grassmannians $G(k,n)$ by first treating the case $k=2$, then moving to $k=3$.

Section \ref{section:other cominuscule} is devoted to other cominuscule spaces like spinor varieties or the Cayley plane. In Section \ref{section:other lines} we look at some other varieties: symplectic and orthogonal Grassmannians of lines, which are not cominuscule,  and products of projective spaces, which are cominuscule with higher Picard rank. Finally, in Section \ref{section:further} we discuss some further directions and open problems. The Appendix is devoted to some technical lemmas needed for our treatment of $G(3,n)$.

\subsection*{Acknowledgments}
This research is part of the CAPES/COFECUB project \emph{Moduli spaces in algebraic geometry and applications}, Capes reference number 88887.191919/2018-00. 
V. B. and D. F. partially supported by the EIPHI Graduate School (contract ANR-17-EURE-0002), Fano-HK ANR-20-CE40-0023 and BRIDGES ANR-21-CE40-0017.
This study was financed in part by the Coordenação de Aperfeiçoamento de Pessoal de Nível Superior - Brasil (CAPES) - Finance Code 001. 
Computations done in Macaulay2 \cite{M2-hyper} and LiE \cite{lie} were indispensable to our investigation.  
We would like to thank Marcos Jardim, Simone Marchesi and Alex Massarenti for fruitful discussions.

\section{Distributions and Foliations} \label{section:distributions and foliations}

Let $X$ be a smooth complex projective manifold of dimension $n$. A \emph{ codimension $p$ distribution $\fol$ on $X$} is a saturated subsheaf $F$ of the tangent sheaf $T_X$ of generic rank $n-p$. The inclusion $F\subset T_X$ induces an exact sequence of the form
\begin{equation}\label{seq:distr}
\fol \colon 0 \longrightarrow F \stackrel{\phi}{\longrightarrow} T_X \stackrel{\eta}{\longrightarrow} N \longrightarrow 0 
\end{equation}
where $N$, called the \emph{normal sheaf} of $\fol$, is a torsion free sheaf
of rank $p$. It follows that $F$, called the \emph{tangent sheaf} of $\fol$,
must be a reflexive sheaf. Two codimension $p$ distributions $\fol$ and $\fol'$ are isomorphic if $\phi(F) = \phi'(F')$ as subsheaves of $T_X$.  

Consider the induced morphism $\wedge^{n-p}\phi^\vee \colon \Omega^{n-p}_X \to \det(F)^\vee$; its image is the ideal sheaf $\IZ$ of a subscheme $Z$ of codimension at least 2 in $X$, twisted by $\det(F)^\vee$; $Z$ is called the \emph{singular scheme} of $\fol$.  The isomorphism $\extp^{n-p}T_X \cong \Omega^{p}_X(\omega_X^\vee)$ tells us that $\wedge^{n-p}\phi^\vee$ defines an element $\wedge^{n-p}\phi^\vee \in H^0(\Omega^{p}_X(\omega_X^\vee\otimes \det(F)^\vee))$. The induced contraction morphism 
\[
\wedge^{n-p}\phi^\vee \colon T_X \to \Omega^{p-1}_X(\omega_X^\vee \otimes \det(F)^\vee)
\]
has $\phi(F)$ as its kernel, hence there exists an isomorphism $\beta\colon \img (\wedge^{n-p}\phi^\vee)  \to N$ such that $\beta \circ \wedge^{n-p}\phi^\vee =  \eta$ in \eqref{seq:distr}.

Conversely let $\omega \in H^0(\Omega^{p}_X(\omega_X^\vee \otimes \det(F)^\vee) )$ not vanishing in codimension one, it defines a codimension $p$ distribution if and only if $F_{\omega} := \ker(\omega \colon T_X \to \Omega^{p-1}_X(\omega_X^\vee))$ has (generic) rank $n-p$. Due to \cite[Proposition 1.2.1]{medeiros:LDS}, ${\rm rk} F_\omega = n-p$ if and only if  for every (closed) point $x \in X \setminus |Z|$ there exists an (affine) open neighborhood $U$ of $x$ and $\alpha_1,\dots, \alpha_p\in H^0(\Omega^1_U)$ such that $\omega$ decomposes as
\begin{equation}\label{eq:LDS}
    \omega|_U = \alpha_1 \wedge \dots  \wedge \alpha_p.
\end{equation}
Such $p$-forms are called \emph{locally decomposable off the singular set}  --  LDS for short. 

Therefore we get a set-theoretical bijection, for each $L \in \Pic(X)$,
\[
\left\{\text{\parbox{6cm}{ $F\hookrightarrow T_X$  saturated subsheaf of rank $n-p$, $\det(F) = L^\vee \otimes \omega_X^\vee $}} \right\} 
\longleftrightarrow 
\left\{ \text{\parbox{5.2cm}{$[\omega] \in \PP(H^0(\Omega^{p}_X(L) )^\vee)$ LDS not vanishing in codimension one}} \right\}
\]
that to $\phi\colon F \to T_X$ associates $\omega = \wedge^{n-p}\phi^\vee$ and to $\omega$ associates $F_\omega$ as above. Then we define the algebraic set
\[
\Dist_p(X,L) := \{ [\omega] \in \PP( H^0(\Omega_X^p(L))^\vee ) \mid \omega \text{ is LDS, } \codim\sing(\omega) \geq 2  \}
\]
that parameterizes codimension $p$ distributions with fixed determinant $\det(F) = L^\vee \otimes \omega_X^\vee$. We remark that $\Dist_p(X,L)$ is not a moduli space in the sense that it does not represent a functor parameterizing distributions. But it can be stratified into moduli spaces by fixing the Hilbert polynomial of $F$, see for instance \cite{quall,distdeg1}. 

A distribution is called \emph{integrable} if it defines a \emph{foliation}, which means that for each $x \in X \setminus |Z|$ there exists an unique analytic immersed subvariety $S \hookrightarrow X$ passing through $x$ such that ${T_S}_x = F\otimes \OO_{S,x}$. Due to a theorem of Frobenius, integrability of $\fol$ is equivalent to $[F,F]\subset F$, where $[\cdot,\cdot]$ is the Lie bracket on $T_X$. In terms of differential forms this integrability condition reads locally as
\begin{equation}\label{eq:integrablity}
    d\alpha_j \wedge \alpha_1 \wedge \dots  \wedge \alpha_p = 0
\end{equation}
where $\alpha_j$ are the $1$-forms from \eqref{eq:LDS}. If integrability holds we simply say that $\fol$ is a foliation. Thus we define the (quasi-projective) algebraic set
\[
\Fol_p(X,L) := \{[\omega] \in \Dist_p(X,L) \mid \omega \text{ is integrable} \}
\]
whose points correspond to foliations on $X$ with determinant $\det(F) = L^\vee \otimes \omega_X^\vee$. 

\begin{Notation}
In our cases of interest $X$ will be embedded in some projective space and we will denote $\OO_X(1) = \OO_{\PP^n}(1)|_X$. Hence we will write $\Fol_p(X,l) := \Fol_p(X,\OO_X(l))$ and when $p=1$ we may just write $\Fol(X,l) := \Fol_1(X,\OO_X(l))$; the same for $\Dist$ and $\IF$ that will be defined below. Moreover, when $X = \PP^n$ it is common to write $L = \OO_{\PP^n}(d+p+1)$, the integer $d$ is called the \emph{degree} of the foliation. Then  $\Fol(\PP^n,d+2)$ is the space of codimension one degree $d$ foliations on $\PP^n$. 
\end{Notation}

For our purposes it will be useful to define also the set of general integrable forms
\[
\IF_p(X,L) := \{ [\omega] \in \PP( H^0(\Omega_X^p(L))^\vee ) \mid \omega \text{ is integrable} \}.
\]
Notice that $\Fol_p(X,L) = \IF_p(X,L) \cap \Dist_p(X,L)$. Moreover, if $H^0(\Omega_X^p(L-D)) = 0$ for every $D\neq 0$ effective divisor, then $\Fol_p(X,L) = \IF_p(X,L)$. 

\subsection{Distributions and foliations on a projective space}
Let $V$ be a (finite dimensional) complex vector space and let ${\PP(V)}$ the associated projective space (of one dimensional quotients); in particular $V = H^0(\OO_{\PP(V)}(1))$. We will write $\PP^n := \PP(V)$, $n+1= \dim V$, unless we need to specify $V$. In the later sections we will be interested in describing distributions and foliations under the action of a semisimple linear algebraic group $\cG$, so that $V$ will be a $\cG$-module. In order to do so we first establish a dictionary between the language of differential forms and representations of $\SL(V)$. For the representation theory of $\SL(V)$ and in particular Schur functors we refer to \cite{weyman:tract}.

Distributions and foliations on ${\PP^n}$ can be described via homogeneous polynomial differential forms. Recall the Euler sequence:
\[
0 \longrightarrow \Omega_{{\PP^n}}^1 \longrightarrow \OO_{{\PP^n}}(-1)\otimes V \longrightarrow \OO_{{\PP^n}} \longrightarrow 0.
\]
Taking exterior powers, twists and global sections we get 
\begin{equation}\label{eq:contrpk}
     H^0\!\left(\Omega_{{\PP^n}}^p(d+p+1)\right) \hookrightarrow S^{d+1}V \otimes \extp^p  V  \stackrel{\iota_R}{\longrightarrow} S^{d+2}V \otimes \extp^{p-1}  V
\end{equation}
where $S^{k}V$ is the $k$-th symmetric power of $V$. We have written $V$ for the vector space generated by the homogeneous coordinates $\{x_j\}$ and for the space of their differentials $\{dx_j\}$. Similarly $H^0(T_{{\PP^n}}(-1)) = V^\vee$ is generated by the (rational) vector fields $\frac{\partial}{\partial x_j}$. The map $\iota_R$ is the contraction with the radial vector field $R = \sum_{j } x_j\frac{\partial}{\partial x_j}$ and it can be written as the composition
\begin{equation}\label{eq:decrad}
    \iota_R \colon S^{d+1}V \otimes \extp^p  V \xrightarrow{1\otimes \Delta}
S^{d+1}V \otimes  V \otimes \extp^{p-1}  V
\xrightarrow{m\otimes 1}
S^{d+2}V\otimes \extp^{p-1}  V,
\end{equation}
where $m\colon S^{d+1}V \otimes  V \to S^{d+2}V$ is the multiplication map $m( p(x) \otimes q(x) ) =  p(x)q(x)$ with $p,q$ homogeneous polynomials of degrees $d+1$ and $1$ respectively; and $\Delta \colon  \extp^p  V \to V \otimes \extp^{p-1} V$ is the diagonal map given by 
\[
\Delta(dx_{i_1} \wedge \dots \wedge dx_{i_p} ) = \sum_{j=1}^p(-1)^{j+1} x_{i_j}\otimes dx_{i_1} \wedge \dots \wedge \widehat{ dx_{i_j}} \wedge \dots \wedge dx_{i_p}. 
\]
Therefore the contraction map with the radial vector field is defined by
\[
\iota_R(Pdx_{i_1} \wedge \dots \wedge dx_{i_p}) = P\sum_{j=1}^p(-1)^{j+1} x_{i_j}dx_{i_1} \wedge \dots \wedge \widehat{ dx_{i_j}} \wedge \dots \wedge dx_{i_p}. 
\]

From this discussion we get that $\omega \in H^0(\Omega_{{\PP^n}}^p(d+p+1))$ may be represented by a (unique) homogeneous polynomial differential $p$-form 
\[
\omega = \sum A_{i_1, \dots, i_p}dx_{i_1} \wedge \dots \wedge dx_{i_p}, \text{ such that } \iota_R\omega = 0.
\]
On the other hand, computing the Lie derivative of $\omega$ with respect to $R$ gives $(d+p+1)\omega$ because of homogeneity. Hence
\[
\omega = \frac{1}{d+p+1}\mathcal{L}_R(\omega) = \frac{1}{d+p+1}(\iota_R d\omega + d\iota_R\omega) = \iota_R \frac{1}{d+p+1} d\omega
\]
and $H^0(\Omega_{{\PP^n}}^p(d+p+1))$ can also be seen as the image of 
\begin{equation}\label{eq:imgcntrctn}
    \iota_R \colon S^{d}V \otimes \extp^{p+1} V \longrightarrow  S^{d+1}V \otimes \extp^{p} V.
\end{equation}

Working with polynomial differential forms also simplifies the verification of the LDS and integrability conditions \eqref{eq:LDS} and \eqref{eq:integrablity}, since it can be done globally. Indeed, owing to \cite{medeiros:LDS}, $\omega$ is LDS if and only if
\begin{equation*}
    (\iota_u\omega)\wedge \omega = 0  \quad \forall u\in \extp^{p-1}V^\vee
\end{equation*}
and it is integrable if in addition
\begin{equation*}
    (\iota_u\omega)\wedge d\omega = 0  \quad \forall u\in \extp^{p-1}V^\vee.
\end{equation*}

\begin{Notation}
We denote by $\Gamma^\lambda V$ the Schur functor of a decreasing sequence of integers $\lambda=(\lambda_1,\cdots,\lambda_k)$ applied to the vector space $V$. For instance $\Gamma^{(q)} V = S^q V$ is the symmetric power, while $\Gamma^{(1,\cdots,1)}V = \extp^k V$ is the exterior power. Moreover we write $a^b$ in a partition if $a$ appears $b$ times, e.g.,  $(1^k) = (1,\cdots,1)$.
\end{Notation}

\begin{Remark}\label{rem:schurformsPn} 
The sequence \eqref{eq:contrpk} is $\SL(V)$-equivariant, hence we can describe the space 
$H^0(\Omega_{{\PP^n}}^p(d+p+1))$ in terms of irreducible representations, i.e. Schur functors applied to $V$. Indeed, from \eqref{eq:decrad} we deduce that 
\[
H^0\!\left(\Omega_{{\PP^n}}^p(d+p+1)\right)=\Gamma^{(d+1, 1^p)}V;
\]
in particular it is always an irreducible $\SL(V)$-module. This is established using \cite[example 2.1.17 (h)]{weyman:tract} (and noting that $\Gamma^{(d+1, 1^p)}V = (K_{(d+1, 1^p)}V^\vee)^\vee = L_{(p+1,1^r)}V$ in Weyman's notation). 
\end{Remark}

\subsection{The space of integrable 1-forms}
Hereafter we fix $p=1$ and let $X$ be a smooth complex projective variety. Let $L \in \pic(X)$ and consider $\omega \in H^0(\Omega_X^1 (L))$. Let $\{ U_\alpha \}_{\alpha\in \Lambda}$ be an open covering that trivializes $L$ and write $\omega = \{ \omega_\alpha \}_{\alpha\in \Lambda}$, where $\omega_\alpha \in \Omega_X^1(U_\alpha)$ are such that, on $U_\alpha \cap U_\beta \neq \emptyset$,  $\omega_\alpha = g_{\alpha \beta} \omega_\beta$ for $\{g_{\alpha \beta} \}_{\alpha,\beta \in \Lambda}$ the cocycle of $L$. The integrability of $\omega$ is measured by the vanishing of $\omega_\alpha \wedge d(\omega_\alpha)$ for each $\alpha$, see \eqref{eq:integrablity}. Note that 
\[
\omega_\alpha \wedge d(\omega_\alpha) = g_{\alpha \beta}^2  \omega_\beta \wedge d(\omega_\beta)
\]
and $\{ \omega_\alpha \wedge d(\omega_\alpha) \}_{\alpha \in \Lambda}$ defines a section in $H^0(\Omega_X^3(2L))$ that, by abuse of notation, we call $\omega\wedge d\omega$. Then we may say that 
\[
\omega \text{ is integrable} \Longleftrightarrow \omega\wedge d\omega = 0.
\]
Thus we get a quadratic map $\psi_X \colon H^0(\Omega_X^1(L)) \to H^0(\Omega_X^3(2L)), \omega \mapsto \omega\wedge d\omega$ whose vanishing locus is the cone over $\IF(X,L)$.  Consider then the associated bilinear map  
\begin{align*}
    \Psi_X\colon S^2 H^0(\Omega_X^1(L)) & \longrightarrow H^0(\Omega_X^3(2L))   \quad, \\ 
    \omega\cdot \eta & \longmapsto  \frac{1}{2}(\omega \wedge d\eta + \eta\wedge d\omega)
\end{align*}
so that $\Psi_X(\omega\cdot\omega) = \psi_X(\omega)$, 
and notice that $S^2 H^0(\Omega_X^1 (L))^\vee = H^0(\OO_{\PP(H^0(\Omega_X^1 (L))^\vee)}(2))$. Dualizing we get that 
\begin{equation}\label{eq:quadricX}
  (\img \Psi_X)^\vee \hookrightarrow S^2 H^0(\Omega_X^1 ( L))^\vee  
\end{equation}
is precisely the truncation in degree $2$ of the homogeneous ideal of $\IF(X,L)$. 

\begin{Lemma}\label{lem:genbyquadrics}
Let $X$ be a smooth complex projective variety and let $L\in \pic(X)$. Then the space of integrable $1$-forms $\IF(X,L)$ is defined by the quadratic polynomials from \eqref{eq:quadricX}. 
\end{Lemma}

\begin{proof}
By construction $(\img \Psi_X)^\vee$ generates the whole ideal of integrable $1$-forms. 
\end{proof}

In the case of projective spaces we can say more, $\Psi_{{\PP^n}}$ is always surjective. 

\begin{Lemma}\label{lem:surjectPn}
Let $V \cong \CC^{n+1}$, $n\geq 3$, and let $l\geq 2$. Then 
\[
\Psi_{{\PP^n}}\colon S^2 H^0(\Omega_{{\PP^n}}^1(l))\longrightarrow H^0(\Omega_{{\PP^n}}^3(2l))
\]
is surjective. As a consequence, the ideal of $\IF({\PP^n},l)$ is generated by the quadratic polynomials given by 
\[
H^0(\Omega_{{\PP^n}}^3(2l))^\vee \hookrightarrow S^2 H^0(\Omega_{{\PP^n}}^1(l))^\vee
\]
\end{Lemma}

\begin{proof}
Since pullbacks commute with exterior differential and exterior product of differential forms, we see that $\Psi_{{\PP^n}}$ is $\SL(V)$-equivariant. From Remark \ref{rem:schurformsPn}, $H^0(\Omega_{{\PP^n}}^3(2l))$ is an irreducible $\SL(V)$-module and this implies, by Schur's Lemma, that $\Psi_{{\PP^n}}$ is either surjective or the zero map. In order to show that it is surjective, it is enough to provide an element $\omega\in H^0(\Omega^1_{{\PP^n}}(l))$ such that $\omega\wedge d\omega \neq 0$.

Consider $\{x_0,\dots,x_n\}$ a basis of $V$ and define
\[
\omega = x_0^{l-2}(x_0dx_1-x_1dx_0 + x_2dx_3-x_3dx_2).
\]
It is clear that 
\[
\begin{split}
    \omega\wedge d\omega = x_0^{2l-4}(x_0dx_1\wedge dx_2\wedge dx_3 - x_1dx_0\wedge dx_2\wedge dx_3 + \\  x_2dx_0\wedge dx_1\wedge dx_3 - x_3dx_0 \wedge dx_1 \wedge dx_2) \neq 0,
\end{split}
\]
concluding the proof.
\end{proof}

\begin{Remark}\label{rem:transtorep}
In the special case $l=2$ it follows from \eqref{eq:imgcntrctn} that we have the diagram  
\[
\begin{tikzcd}
S^2 \extp^2 V \arrow[d, "\iota_R"] \arrow[r, "\Psi_{{\PP^n}}"] & \extp^4 V \arrow[d,"\iota_R"]\\
S^2 H^0(\Omega_{{\PP^n}}^1(2)) \arrow[r,"\Psi_{{\PP^n}}" ] &  H^0(\Omega_{{\PP^n}}^3(4))  
\end{tikzcd}
\]
whose vertical arrows are isomorphisms due to Remark \ref{rem:schurformsPn}. Then there exists a unique choice for the top horizontal arrow in order to make this diagram commutative: $u \cdot v \mapsto u\wedge v$. By abuse of notation we also call this map $\Psi_{{\PP^n}}$. We remark that in this case $(\img \Psi_{{\PP^n}})^\vee = \extp^4 V^\vee$ is generated by the Pl\"ucker relations. Then it follows from Lemma \ref{lem:genbyquadrics} that $\Fol({\PP^n},2)\cong G(2,V^\vee)$. This is essentially the argument of \cite[Theorem 4.3]{ACM:FanoDist}; in fact, they use the more general map $S^k \extp^2 V \mapsto \extp^{2k}V$ to get an isomorphism between the $k$-th secant variety of $G(2,V^\vee)$ and the space of so-called \emph{class $k$ distributions}, class $0$ meaning integrable. 
\end{Remark}

\subsection{Distributions and foliations under restriction} 
\label{subsection:restriction}
Now consider a smooth subvariety $i\colon X \hookrightarrow {\PP^n}$ and denote $\OO_X(1) = i^*\OO_{{\PP^n}}(1)$. The embedding $i$ induces the pullback (restriction) of twisted differential forms $i^\ast_p \colon \Omega_{{\PP^n}}^p(l) \to \Omega_X^p(l)$. This morphism is surjective since we can express $i^\ast_p = \wedge^p i^\ast_1\otimes 1$ and $i^\ast_1$ is the composition of the two surjections appearing in the following sequences:
\begin{gather*}
0 \longrightarrow N_X^{\vee} \longrightarrow \Omega_{{\PP^n}}^1|_X \longrightarrow \Omega_X^1 \longrightarrow 0, 
\\
0 \longrightarrow  \Omega_{{\PP^n}}^1 \otimes \mathcal{I}_X \longrightarrow \Omega_{{\PP^n}}^1 \longrightarrow \Omega_{{\PP^n}}^1|_X \longrightarrow 0.
\end{gather*}
Taking global sections gives us
\begin{equation*}
    i_p^*\colon H^0( \Omega^p_{{\PP^n}}(l)) \longrightarrow H^0(\Omega^p_X(l))
\end{equation*}
for each $p$ and $l$ and it also induces a rational map between projective spaces
\[
\begin{tikzcd}
\pi_p \colon \PP(H^0( \Omega^p_{{\PP^n}}(l))^\vee) \arrow[r,dashrightarrow] &  \PP(H^0(\Omega^p_X(l))^\vee).
\end{tikzcd}
\]
Since pullbacks commute with exterior products and the exterior derivative, the closure of the image of $\IF_p({\PP^n}, l)$ is contained in $\IF_p(X,l)$: 
\[
\overline{\pi_p(\IF_p({\PP^n}, l))} \subset \IF_p(X,l).
\]
For $p\geq 2$ and general $X$ the integrability condition is not easy to check. 
To study it for $p=1$, we write the commutative diagram:
\begin{equation}\label{comm_diag_rest}
    \begin{tikzcd}
        S^2 H^0( \Omega^1_{{\PP^n}}(l)) \arrow[d, twoheadrightarrow, "\Psi_{{\PP^n}}" ] \arrow[r , "S^2i_1^*"]   & S^2 H^0(\Omega^1_X(l)) \arrow[d, "\Psi_X"] \\
        H^0( \Omega^3_{{\PP^n}}(2l))   \arrow[r,"i_3^*" ] & H^0(\Omega^3_X(2l)).
    \end{tikzcd}
\end{equation}
To simplify notation we write :
\[
\pi=\pi_1 : \PP(H^0( \Omega^1_{{\PP^n}}(l))^\vee) \dashrightarrow  \PP(H^0(\Omega^1_X(l))^\vee)
\]
and:
\[
\mathbb{W}_X^l :=\PP(H^0(\Omega^1_X(l))^\vee).
\]
Then $\IF(X, l) \subset \mathbb{W}_X^l$ and the rational map we are concerned with is $\pi \colon \mathbb{W}_{\PP^n}^l \dashrightarrow \mathbb{W}_X^l$. Hence the dual of \eqref{comm_diag_rest} can be written as 
\begin{equation}\label{comm_diag_dual}
    \begin{tikzcd}
       H^0(\OO_{\mathbb{W}_{\PP^n}^l}(2))   & H^0(\OO_{\mathbb{W}_X^l }(2))  \arrow[l, "(S^2i_1^*)^\vee"] \\
        H^0( \Omega^3_{{\PP^n}}(2l))^\vee \arrow[u, hookrightarrow, "\Psi_{{\PP^n}}^\vee" ]    & H^0(\Omega^3_{X}(2l))^\vee \arrow[l, "(i_3^*)^\vee"] \arrow[u, "\Psi_X^\vee"].
    \end{tikzcd}
\end{equation}

If $i_1^*$ is surjective then every twisted $1$-form in $H^0(\Omega^1_X(l))$ is the pullback of an element of $H^0( \Omega^1_{{\PP^n}}(l))$. This raises the question of whether the integrable $1$-forms on $X$ as well come from integrable $1$-forms on ${\PP^n}$. In other words, when do we have the equality
\[
\overline{\pi(\IF({\PP^n}, l))}  =  \IF(X,l) ?
\]

\begin{Lemma}\label{lem:surjideals2}
Let $X$ be a smooth complex projective variety and let $l\geq 2$ be an integer. Assume there exists an embedding $i\colon X \hookrightarrow {\PP^n}$ such that $i_1^*\colon H^0( \Omega^1_{{\PP^n}}(l)) \to H^0(\Omega^1_X(l))$ and $i_3^*\colon H^0( \Omega^3_{{\PP^n}}(2l)) \to H^0(\Omega^3_X(2l))$ are surjective. Then 
\begin{enumerate}
    \item $\Psi_X$ is surjective and the ideal of $\IF(X,l)$ is generated by the quadratic equations
\[
\Psi^\vee_X(H^0(\Omega^3_X(2l))^\vee)\subset S^2 H^0(\Omega^1_X(l))^\vee.
\]
    \item \label{27ii} If the square \eqref{comm_diag_dual} is cartesian then we have equality on the degree $2$ part of the ideals:
\[
({I_{\overline{\pi(\IF({\PP^n}, l))}}})_2  = (I_{\IF(X,l)})_2.
\]
{
\item \label{27iii}
If $i_1^*$ and $i_3^*$ are isomorphisms, then 
\[
\IF(X,l) \cong \IF(\PP^n, l).
\]
}
\end{enumerate}
\end{Lemma}


Note that if $\overline{\pi(\IF({\PP^n}, l))}$ is defined by quadrics and \ref{27ii} holds, then we get an isomorphism $\IF({\PP^n}, l) \cong \IF(X,l)$. 
However it is not clear whether $\overline{\pi(\IF({\PP^n}, l))}$ is defined by quadrics even though $\IF({\PP^n}, l)$ is. The image of a linear projection may acquire higher degree generators, see for instance \cite[Example 4.3]{AR:projvar}. 


\begin{proof}[Proof of Lemma \ref{lem:surjideals2} ]
Notice that since $i_1^*$ is surjective, so is $S^2i_1^*$. Adding that $\Psi_{{\PP^n}}$ is surjective, due to Lemma \ref{lem:surjectPn}, we see that $\img \Psi_X = \img i_3^* = H^0(\Omega^3_X(2l))$. 

From Lemma \ref{lem:genbyquadrics} we know that the vertical arrows of \eqref{comm_diag_dual} are the inclusions of the ideals of $\IF({\PP^n}, l)$ and $\IF(X, l)$, proving our first assertion. {If we assume moreover that $i_1^*$ and $i_3^*$ are isomorphisms, then we get that $\IF(X,l)$ and $\IF(\PP^n, l)$ are defined by the same homogeneous ideal; this proves \ref{27iii}.}

Now we prove \ref{27ii}. For simplicity, write $Y:= \IF({\PP^n}, l)$ and $Z := \IF(X,l)$. By definition, 
\[
(I_{\overline{\pi(Y)}})_2 = (I_{Y})_2 \cap H^0(\OO_{\PP(\mathbb{W}_X^l)}(2)) ,
\]
i.e. $(I_{\overline{\pi(Y)}})_2$ is the pullback of $\Psi_{{\PP^n}}^\vee$ by $S^2i_1^*$. By the universal property of pullback diagrams, \eqref{comm_diag_dual} is cartesian if and only if the natural inclusion 
\[
H^0(\Omega^3_X(2l))^\vee = (I_{Z})_2 \subset (I_{\overline{\pi(Y)}})_2
\]
is an equality. 
\end{proof}

\begin{Proposition}\label{prop:restinject}
Let $X$ be a smooth projective variety such that $\pic(X) = \ZZ$ with a very ample generator $\OO_X(1)$. Consider the embedding $i\colon X \hookrightarrow \PP(H^0(\OO_X(1)))$. Then the restricted pullback map
\[
\pi \colon \Fol(\PP(H^0(\OO_X(1))), 2) \longrightarrow \Fol(X, 2)
\]
is an embedding.
\end{Proposition}


\begin{proof}
Fix $\PP := \PP(H^0(\OO_X(1)))$, then we have $H^0(\PP, \OO_\PP(1)) = H^0(\OO_X(1))$. Now denote $W := \ker (i_1^*) \subset H^0(\Omega_{\PP}^1(2))$. To show that the restriction of $\pi$ to $\Fol(\PP,2)$ is an embedding it is enough to show that each fibre of $\pi$ intersects $\Fol(\PP, 2)$ at precisely one (reduced) point. Since $\Fol(\PP,2)$ is smooth and irreducible, this is equivalent to show that $\PP(W^\vee)$ does not intersect the first secant variety: $\PP(W^\vee)  \cap {\rm Sec }_1(\Fol(\PP,2)) =\emptyset$. 

Now recall that $\Fol(\PP(V),2) \cong G(2,V^\vee)$ in its Pl\"ucker embedding and ${\rm Sec }_1(G(2,V^\vee))$ correspond to skew-symmetric matrices of rank $\leq 4$. Then $\pi$ is an embedding if and only if every element in $W \subset \extp^2 V$ has rank $\geq 6$; that is what we are going to prove.

First note that there is no element in $W$ of rank two. Indeed any such element has the form $\omega = fdg-gdf \in H^0(\Omega_{\PP}^1(2))$ where $f,g \in H^0(\PP, \OO_\PP(1))$. If $\omega \in W$, i.e. $i_1^*(\omega) = 0$, then there exist constants $a,b \in \CC$, not both zero, such that $af+bg$ vanishes on $X$, but that is not possible since $H^0(\PP, \OO_\PP(1)) = H^0(\OO_X(1))$.

Now assume that there exists $\omega \in W$ of rank four. Then we may write it as $\omega = fdg-gdf+ f'dg'-g'df'$ for some $f,g,f',g' \in H^0(\PP, \OO_\PP(1))$ linearly independent. From $i_1^*(\omega) = 0$ we get $i_1^*(fdg-gdf) = -i_1^*(f'dg'-g'df')$ hence there exists a Möbius transformation $\tau$ such that $\frac{f'}{g'}|_X = \tau(\frac{f}{g}|_X)$. Hence, up to composing with $\tau^{-1}$, we may assume that $\frac{f'}{g'}|_X = \frac{f}{g}|_X$, which means
\[
fg'-f'g \in H^0(\PP, \mathcal{I}_X(2)).
\]
Observe that $fg'-f'g$ is a quadratic polynomial of rank $4$; here we mean the rank of the associated symmetric (Hessian) matrix. To conclude we claim that under our hypothesis there exists no quadric of rank $\leq 4$ in the ideal of $X$. We will prove this claim in Lemma  \ref{lem:rk4quad}. 

{ To conclude the proof we need to show that the image of $\pi|_{\Fol(\PP,2)}$ is contained in $\Fol(X,2)$. Let $\omega = fdg-gdf$ represent an element of $\Fol(\PP,2)$. Clearly, $i_1^\ast\omega$ is integrable, and nonzero by the argument above. Then we only need to show that $i_1^\ast\omega$ does not vanish in codimension one.  

Note that $i_1^\ast\omega$ corresponds to the pencil $\Lambda = \{ af_X+bg_X = 0 \}_{(a:b)\in \PP^1}$ so that the zero locus of $i_1^\ast\omega$ has support on the base points of $\Lambda$ plus the singularities of the fibers. Then $i_1^\ast\omega$ vanishes in codimension one only if $\Lambda$ has either a fixed part or a multiple fiber. Either way we get nontrivial effective divisors $D, D' >0$ such that $\OO_X(D'+ D) = \OO_X(1)$, which yields an absurd since $\OO_X(1)$ generates $\Pic(X)$. 
}
\end{proof}

\begin{Remark}
Given a general $\omega \in H^0(\Omega_{\PP}^1(2)$ we may decompose it as $\omega = \eta_1+ \dots + \eta_r$ with $\eta_j$ integrable. Proposition \ref{prop:restinject} implies that the restriction to $X$ of each $\eta_j$ only vanishes in codimension at least two. However the same may not be true for $\omega$ as we will see in the proof of Proposition \ref{prop_OG_equations}. 
\end{Remark}

\begin{Lemma}\label{lem:rk4quad}
Let $X$ be a smooth projective variety such that $\pic(X) = \ZZ$ with a very ample generator $\OO_X(1)$. Consider the embedding $i\colon X \hookrightarrow \PP =  \PP(H^0(\OO_X(1)))$. Then there is no element in $H^0(\PP, \mathcal{I}_X(2))$ of rank $\leq 4$.
\end{Lemma}

\begin{proof}
Assume, aiming at a contradiction, that there exists $q \in H^0(\PP, \mathcal{I}_X(2))\setminus \{0\}$ of rank $\leq 4$. Then, there exist linear polynomials $a,b,c,d\in H^0(\PP,\OO_\PP(1))$ such that $q = ad-bc$. Then define the morphism
\[
A\colon \OO_X(-1)^{\oplus 2} \xrightarrow{\begin{pmatrix} a|_X & b|_X \\ c|_X & d|_X
\end{pmatrix}} \OO_X^{\oplus 2}
\]
whose generic rank is equal to one since $q \in H^0(\PP, \mathcal{I}_X(2))$ and $H^0(\PP, \mathcal{I}_X(1)) = 0$. 
Hence $\img A$ is a rank one torsion free sheaf. Since $X$ is smooth (integral and locally factorial) $\img A = \IZ(l)$ for some subscheme $Z\subset X$ of codimension $\geq 2$ and $l\in\ZZ$. Moreover, $\ker A = \OO_X(-2-l)$. From $\OO_X(1)$ being a generator of $\pic(X)$
we deduce that $l \in \{-1,0\}$ and we have two possibilities:
\begin{enumerate}
    \item $l=-1$: Then $\ker A = \OO_X(-2-l) = \OO_X(-1)$ and the inclusion $\ker A \hookrightarrow \OO_X(-1)^{\oplus 2}$ is given by a constant vector $(u,v) \in H^0(\OO_X)^{\oplus 2}$. Up to multiplying $A$ on the left with some element of $\GL(2, \CC)$ we can assume that $(u,v) = (1,0)$, thus
    \[
    \begin{bmatrix}
    a|_X \\c|_X
    \end{bmatrix}
    = \begin{pmatrix} a|_X& b|_X \\ c|_X & d|_X
\end{pmatrix} \begin{bmatrix}
    1 \\0
    \end{bmatrix} = \begin{bmatrix}
    0 \\0
    \end{bmatrix}.
    \]
    Since $H^0(\PP,\OO_\PP(1)) = H^0(\OO_X(1))$ we have that $a=c =0$ whence $q=0$.
    \item $l=0$: Then $\img A = \IZ$ and the inclusion $\IZ\hookrightarrow \OO_X^{\oplus 2}$ must factor as
    \[
    \IZ \hookrightarrow \OO_X \hookrightarrow \OO_X^{\oplus 2},
    \]
    and the map on the right is given by a constant vector
    that we may assume to be $(1,0)$. It follows that $c|_X = d|_X =0 $ and, as in the previous case, this implies $q=0$.
\end{enumerate}

\end{proof}

We remark that we did not use that $X$ is smooth but only that a rank one reflexive sheaf is locally free. This holds for $X$ an integral and locally factorial variety, see \cite[Proposition 1.9]{hartshorne:stable-reflexive}.

\section{Integrable forms on homogeneous spaces}
\label{section:homogeneous spaces}

We will focus now on rational homogeneous varieties, i.e., varieties which admit a transitive action of a linear Lie group. Using this action it will be possible, in some cases, to understand what are the codimension one minimal degree foliations on these varieties. Let us give a brief introduction before studying their $1$-forms and foliations; we refer to \cite{ottaviani} for more details.

Let $\cG$ be a  semisimple linear Lie group over $\CC$ and $X$ a $\cG$-homogeneous variety. By the transitivity assumption, the stabilizers of all the points in $X$ are conjugated to some subgroup $\cP\subset \cG$. The variety $X$ is projective if and only if $\cP$ is  \emph{parabolic}, i.e. $\cP$ contains a Borel (maximal connected solvable algebraic) subgroup of $\cG$. Sometimes we will write $X=\cG / \cP$ to make the group and the parabolic subgroup explicit. Moreover $\Pic(X) \cong \ZZ$ if and only if the parabolic subgroup $\cP$ is maximal for the inclusion; we will call such varieties \emph{generalized Grassmannians}.


\begin{Example}
\label{ex_sl_n}
Let us fix $\cG=\SL(n+1)$. Then a Borel subgroup $\cB\subset \cG$ is given by upper triangular matrices. A parabolic subgroup $\cP$ containing $\cB$ can be written as a product $\cP=\cD\cB$, where $$\cD:=\{(g_1,\cdots,g_{k+1})\in \GL_{i_1}\times \cdots \times \GL_{i_k-i_{k-1}}\times \GL_{n+1-i_k}\mid \det(g_1)\cdots\det(g_{k+1})=1\}$$ is a subgroup of block diagonal matrices of fixed size given by an integer sequence $1\leq i_1 < \cdots < i_k \leq n$. The quotient $X =\cG / \cP$ is the Flag variety $\Fl(i_1,\cdots,i_k,n+1)$ parameterizing flags $[\CC^{i_1}\subset \cdots \subset \CC^{i_k}\subset \CC^{n+1}]$. The quotient $\cG/\cB$ is thus the complete flag variety $\Fl(1,2,\cdots,n+1)$, while Grassmannians and projective spaces are obtained by setting $k=1$ and correspond to maximal parabolic subgroups. 
\end{Example}

Fix  $\cG$  a  semisimple linear Lie group over $\CC$ and let $\alpha_1\ldots,\alpha_r$ be its set of fundamental roots. An irreducible representation of $\cG$ is uniquely determined by its dominant weight, which has the form $\lambda=\sum_{i=1}^r a_i \lambda_i$, where $\{\lambda_1,\ldots,\lambda_r\}$ is the set of fundamental weights of $\cG$, and $a_i \in \ZZ_{\geq 0}$; such a representation will be denoted by $V_\lambda$. A parabolic subgroup $\cP$ is defined by the choice of a subset of simple roots of $\cG$; moreover $\cP$ is maximal if it corresponds to a single root. In Example \ref{ex_sl_n}, the parabolic $\cP$ such that $\SL(n+1)/\cP=\Fl(i_1,\cdots,i_k,n+1)$ is defined by the choice of the subset $\{\alpha_{i_1},\ldots,\alpha_{i_k}\}$ of simple roots, while the Grassmannian $G(k,n+1)$ is defined by the choice of the root $ \alpha_k$. 

A vector bundle $E$ on $X$ is called \emph{homogeneous} if the action of $\cG$ on $X$ lifts to a compatible action on $E$. There exists an equivalence of categories between homogeneous vector bundles on $X=\cG / \cP$ and representations of $\cP$: given a representation $V$ of $\cP$ one constructs a vector bundle $E:=\cG\times_\cP V$, and given a vector bundle $E$, its fiber  over the point stabilized by $\cP$ affords a $\cP$-representation $V$; see \cite[\S 9]{ottaviani} for details.


Recall that a $\cP$-representation $V$ is called \emph{irreducible} if it has no proper nontrivial subrepresentation $W\subset V$. Also, a representation is called \emph{completely reducible} if it decomposes as a direct sum $V = \bigoplus_i V_i$ where each $V_i$ is an irreducible representation of $\cP$. Moreover, a $\cP$-representation has a filtration with completely reducible factors. We carry these adjectives to the vector bundles associated with the representations. In particular, any homogeneous bundle $E$ on $X$ admits a filtration with completely reducible factors.

Any irreducible representation of $\cP$ is uniquely determined by its restriction to the semisimple part of $\cP$, hence it corresponds to a $\cP$-dominant weight $\lambda$. When  $\cP$ is maximal defined by the $k$-th root $\alpha_k$, the set of $\cP$-dominant weights is precisely the set of weights for $G$ of the form $\lambda=\sum_{i=1}^r a_i \lambda_i$, where $a_i \in \ZZ$ and $a_i \ge 0$ for $i \ne k$.

The irreducible bundle associated with the weight $\lambda$ will be denoted by $E_\lambda$. The Bott-Borel-Weil Theorem, \cite{bott:homogeneous}, says that $H^0(E_\lambda) \cong V_\lambda$, the $\cG$-representation associated with highest weight $\lambda$. Moreover, $V_\lambda\neq 0$ if and only if $a_k \ge 0$, i.e., $\lambda$ is also dominant with respect to $\cG$.

We use the following convention. We set $\lambda_0=\lambda_{r+1}=0$ and, given a formal expression $\lambda=\sum_{i\in \ZZ}
a_i \lambda_i$, we put $E_\lambda=0$ and $V_\lambda=0$ if there is $j \ne k$ such that $a_j<0$, or if $a_j \ne 0$ for some $j<0$ or $j>r+1$. For instance we have:
\[
S^2 V_{\lambda_k} \cong \bigoplus_{j\in \ZZ_{\geq 0}} V_{\lambda_{k-2j}+\lambda_{k+2j}}.
\]

We start with the following elementary result.

\begin{Lemma}\label{lem:surjGeneral}
Let $X \subset { Y}$ be an embedding of homogeneous $\cG$-varieties. Consider $M$ a homogeneous $\cG$-bundle on ${ Y}$ and $N$ a homogeneous $\cG$-bundle on $X$ such that 
\begin{enumerate}
    \item $N$ is completely reducible;
    \item $M$ is globally generated;
    \item there exists a surjective ($\cG$-equivariant) morphism $M \twoheadrightarrow N$.
\end{enumerate}
Then the induced morphism on global sections $H^0(Y,M) \to H^0(X,N)$ is also surjective.
\end{Lemma}

\begin{proof}
From the hypotheses we have a surjective map
\[
\begin{tikzcd}
H^0({ Y}, M)\otimes \OO_{{ Y}} \arrow[r, twoheadrightarrow ] & M \arrow[r, twoheadrightarrow ] & N   = \bigoplus_\mu E_\mu.
\end{tikzcd}
\]
Up to composing further with a projection $N \twoheadrightarrow E_\mu$ we may assume $N= E_\mu$ irreducible and we get an equivariant map
\[
\begin{tikzcd}
H^0(Y, M)\otimes \OO_{Y} \arrow[r, twoheadrightarrow ] & E_\mu
\end{tikzcd}
\]
that must factor through $H^0(X, E_\mu) \otimes \cO_X = V_\mu \otimes \cO_X$. In particular, the induced map on global sections is not zero and the surjectivity follows from Schur's Lemma.
\end{proof}

\subsection{Cominuscule varieties and cominuscule Grassmannians}\label{subsection:cominuscule}

A class of rational homogeneous varieties suitable to our study of foliations under restriction is that of \emph{cominuscule varieties} which are also classically known to the Differential Geometry comunity as \emph{Hermitian symmetric varieties}, whose study goes back to E. Cartan.  For details we refer to \cite{landsberg-manivel:geometry, ottaviani-rubei:quiver}. Any cominuscule variety decomposes as a product $X = X_1 \times \cdots \times X_k$ where each $X_j$ is one of the varieties in Table \ref{tab:comin}; we call these irreducible varieties \emph{cominuscule Grassmannians}, since they are cominuscule varieties of Picard rank one. 

\begin{table}[ht]
    \centering
    \renewcommand{\arraystretch}{1.3}
    \begin{tabular}{|c|c|c|c|}\hline
      Type   & diagram & description & dimension\\ \hline
        $A_r$    & \dynkin[scale=1.5] A{*.*x*.*} & Grassmannian of $k$-planes in $\CC^{r+1}$ & $k(r+1-k)$\\ \hline
        $B_r$    & \dynkin[parabolic=1,scale=1.5] B{} & $(2r-1)$-dimensional quadric hypersurface & $2r-1$ \\ \hline
        $C_r$    & \dynkin[parabolic=16,scale=1.5] C{}  & Grassmannian of Lagrangian $r$-planes in $\CC^{2r}$ & $ \frac{r(r+1)}{2} $\\ \hline
        $D_r$    &  \dynkin[parabolic=1,scale=1.5] D{} & $(2r-2)$-dimensional quadric hypersurface & $2r-2$ \\ \hline
        $D_r$    &  \dynkin[parabolic=32,scale=1.5] D{} & Spinor varieties $OG(r,2r)$ & $\frac{r(r-1)}{2} $\\ \hline
        $E_6$    & \dynkin[parabolic = 1,scale=1.5] E6 & The Cayley plane  & $ 16$ \\ \hline
        $E_7$    & \dynkin[parabolic = 64,scale=1.5] E7 & The Freudenthal variety & $27$\\ \hline 
    \end{tabular}
    \medskip
    \caption{ Cominuscule Grassmannians up to isomorphism}
    \label{tab:comin}
\end{table}

An important feature of a cominuscule variety $X$ is that its tangent bundle $T_X$ is a completely reducible bundle. In fact we could define cominuscule varieties by the complete reducibility of the tangent bundle, see \cite[Corolary 36]{BS:Hochschild}. 

\subsubsection{Cotangent bundles of cominuscule Grassmannians}\label{sssect:cotangComin}
Since the tangent bundle of a cominuscule Grassmannian $X=\cG/\cP$ is completely reducible, it corresponds to $P$-dominant weights and so does $\Omega^1_X$. Knowing these weights will be necessary for further computations hence we will describe them here. Notice that any cominuscule Grassmannian is the quotient of a simple Lie group.

Denote by $\mathfrak{g}$ the Lie algebra of $\cG$ and notice that $H^0(X,T_X)\cong T_{\Aut(X),\id}=T_{\cG,\id}=\mathfrak{g}$. Since $\cG$ is simple, $\mathfrak{g}$ is an irreducible $\cG$-representation and thus corresponds to a highest weight $\delta$, called the highest root. Since $T_X$ is globally generated and completely reducible, it follows from the Borel-Weil-Bott Theorem that $T_X \cong E_\delta$, in particular it is irreducible. For the list of highest roots of simple Lie algebras we refer to \cite[Table 1]{OniVinberg}.

Note that $\Omega^1_X=(E_\delta)^\vee \cong E_\mu$ for some $P$-dominant weight $\mu$. Since this $P$-representation is irreducible, we can restrict to $S_P$ the semisimple part of $P$. Let $\cP$ be the parabolic subgroup associated to the root $\alpha_k$. From the Dynkin diagram of $\cG$ remove the vertex associated to $\alpha_k$, then $S_P$ is the semisimple group associated to this new modified Dynkin diagram. 
Let $\delta^*$ be the dual of $\delta$ seen as weights of $S_P$; it can be computed using the symmetries of the Dynkin diagram of $S_P$, see \cite[p.195]{OniVinberg}. Then $\mu = \delta^* + a\lambda_k$ for some integer $a$ which may be determined by the first Chern class formula:
\begin{equation}\label{eq:c1hom}
    c_1(E_\mu) = {\rm rk}(E_\mu) \frac{\langle \mu, \lambda_k \rangle}{ \langle \lambda_k , \lambda_k \rangle}, 
\end{equation}
see \cite[p.56]{ottaviani} and \cite[\S 2.4.1]{VladThesis}. The pairing $\langle \cdot, \cdot \rangle$ on the space of weights is induced by the Killing form of $\mathfrak{g}$, hence it can be computed using inverse Cartan matrices, listed in \cite[Table 2]{OniVinberg}; we remark that they use a different numbering of roots. Below we compute the weights of $\Omega_X^1$ for each entry of Table \ref{tab:comin}.

\medskip
\paragraph{Classical Grassmannians $A_r/P_k$} In this case $\cG=\SL_{r+1}$ and $\cP$ is associated to a fundamental root $\alpha_k$. The highest root is the sum $\delta=\alpha_1+ \cdots + \alpha_n=\lambda_1+\lambda_r$. Then \eqref{eq:c1hom} reads
\[
c_1(T_X) = c_1(E_\delta) = k(r+1-k) \frac{\langle \lambda_1, \lambda_k \rangle + \langle \lambda_r, \lambda_k \rangle}{ \langle \lambda_k , \lambda_k \rangle} = k(r+1-k) \frac{r+1-k + k}{k(r+1-k)} = r+1,
\]
as expected. Removing the vertex associated to $\alpha_k$, $k\notin \{1 ,r\}$, from the Dynkin diagram $A_r$ gives two disjoint components $A_{k-1}$ and $A_{r-k}$. Thus $\delta^* = \lambda_{k-1} + \lambda_{k+1}$ and imposing $c_1(\Omega^1_X) = -r-1$ we get $a = -2$, hence $\Omega^1_X = E_{\lambda_{k-1} -2\lambda_k + \lambda_{k+1}}$. For $k \in \{1, r\}$, i.e. $X\cong \PP^r$, we arrive at the same conclusion with the convention that $\lambda_0 := 0$. 

\medskip
\paragraph{Odd-dimensional Quadrics $B_r/P_1$} In this case $\cG=\SO_{2r+1}$ and $\cP$ is associated to the fundamental root $\alpha_1$ and the highest root is $\delta = \alpha_1 + 2(\alpha_2 + \dots + \alpha_r)  = \lambda_2$. Removing the vertex associated to $\alpha_1$ we get a diagram of type $B_{r-1}$, hence $\delta^* = \delta = \lambda_2$. Then $\Omega_X^1 = E_{a\lambda_1+ \lambda_2}$ and imposing $c_1(\Omega_X^1) = 1-2r$ we get $a=-2$.

\medskip
\paragraph{Lagrangian Grassmannians $C_r/P_r$} In this case $\cG=\Sp_{2r}$ and $\cP$ is associated to the fundamental root $\alpha_r$ and the highest root is $\delta =  2(\alpha_1 + \dots + \alpha_{r-1}) + \alpha_r  = 2\lambda_1$. Removing the vertex associated to $\alpha_r$ we get a diagram of type $A_{r-1}$, hence $\delta^* = 2\lambda_{r-2}$. Then $\Omega_X^1 = E_{2\lambda_{r-1}+ a\lambda_r}$ and imposing $c_1(\Omega_X^1) = -c_1(T_X) = -r-1 $ we get $a=-2$.

\medskip
\paragraph{Even-dimensional Quadrics $D_r/P_1$} For $r=3$ we have $X\cong G(2,4)$ of type $A$, hence we assume $r\geq 4$. Then $\cG=\SO_{2r}$ and $\cP$ is associated to the fundamental root $\alpha_1$ and the highest root is $\delta = \alpha_1 + 2(\alpha_2 + \dots + \alpha_{r-2}) + \alpha_{r-1} + \alpha_r  = \lambda_2$. Removing the vertex associated to $\alpha_r$ we get a diagram of type $D_{r-1}$, hence $\delta^* = \delta = \lambda_{2}$. Then $\Omega_X^1 = E_{a\lambda_1+ \lambda_2}$ and imposing $c_1(\Omega_X^1) = 2-2r$ we get $a=-2$.

\medskip
\paragraph{Spinor varieties $D_r/P_r$} Again we assume $r\geq 4$, otherwise $X\cong \PP^r$. We still have $\cG=\SO_{2r}$ and $\delta = \lambda_2$ but $\cP$ is now associated to $\alpha_r$. Removing the corresponding vertex gives a diagram of type $A_{r-1}$, thus  $\delta^* = \lambda_{r-2}$. Then $\Omega_X^1 = E_{ \lambda_{r-2}+ a\lambda_r}$ and imposing $c_1(\Omega_X^1) = 2-2r$ we get $a=-2$.

\medskip
\paragraph{The Cayley plane $E_6/P_1$} In this case $\cG$ is the exceptional group of type $E_6$ and $\cP$ is associated to $\alpha_1$. Moreover $\delta = \lambda_2$ and removing the vertex corresponding to $\alpha_1$ we get  $\delta^* = \lambda_3$, thus $\Omega_X^1 = E_{ a\lambda_1 + \lambda_3}$. From \eqref{eq:c1hom} we get $-12 = c_1(\Omega_X^1) = 16a + 20$, hence $a=-2$.

\medskip
\paragraph{The Freudenthal variety $E_7/P_7$} In this case $\cG$ is the exceptional group of type $E_7$ and $\cP$ is associated to $\alpha_7$. Moreover $\delta = \lambda_1$ and $\delta = \lambda_6$ hence $\Omega_X^1 = E_{ \lambda_6 + a\lambda_7}$. Computing $-18 = c_1(\Omega_X^1) = 27a + 36$ we get $a=-2$.

\begin{Remark}
In each case above $\Omega_X^1 = E_{\delta^* -2 \lambda_k}$ and the coefficient of $\lambda_k$ in the expansion of $\delta^*$  is zero. As a consequence, $\delta^* -\lambda_k$ is not $\cG$-dominant hence $H^0(\Omega_X^1(1)) = H^0(E_{\delta^* - \lambda_k}) =  0$ for every cominuscule Grassmannian $X$. It follows from the Bott-Borel-Weil Theorem that $\Omega_X^1(2) = E_{\delta^*}$ is globally generated. 
\end{Remark}

\subsection{Restriction of integrable forms to cominuscule varieties}

Here we prove the main result of this section, dealing with restriction of integrable forms from the ambient space to a cominuscule variety, accounting for item \ref{A-ii} of Theorem \ref{mainA} from the Introduction.
So let $X=\cG/\cP$ be a cominuscule variety, $\lambda$ be a $\cG$-dominant weight. Let us assume that $E_\lambda$ is a very ample line bundle on $X$ and write $i$ for the corresponding $\cG$-equivariant embedding of $X$ into $\PP(V_\lambda)$, so $i^*(\cO_{\PP(V_\lambda)}(1)) = \cO_X(1) = E_\lambda$.
In the notation of \S \ref{subsection:restriction}, the result reads as follows.

\begin{Theorem}
\label{thm:folcomin}
Let $X \subset \PP(V_\lambda)$ be a cominuscule variety. The space of codimension one degree $d$ integrable forms $\IF(X,d+2)\subset \PP(H^0(\Omega^1_X(d+2))^\vee)$ is defined by the quadratic equations given by the $\cG$-equivariant inclusion 
\begin{equation*}
   H^0(\Omega^3_X(2d+4))^\vee \subset S^2 H^0(\Omega^1_X(d+2))^\vee
\end{equation*}
given by $\Psi_X^\vee$ from diagram \eqref{comm_diag_dual}.
\end{Theorem}

We start with the following lemma of independent interest.

\begin{Lemma}\label{lem:surjPbComin}
Let $X \subset \PP(V_\lambda)$ be a cominuscule variety. Then the induced pullback maps
\[
i^\ast_p\colon H^0(\Omega_{\PP(V_\lambda)}^p(d + p+1))  \longrightarrow  H^0( \Omega_X^p(d+p+1))
\]
are surjective for every $p\geq 1$ and $d\geq 0$.
\end{Lemma}

\begin{proof}
Notice that for any $d\geq 0$ the vector bundle $\Omega_{{\PP^n}}^p(d+p+1)$ is globally generated and homogeneous with respect to the action of $\cG \subset \SL(V_\lambda)$. Recall from \S\ref{subsection:restriction} the surjections:
\begin{equation} \label{display:restriction}
\Omega_{{\PP^n}}^p(d+p+1) \to \Omega_{{\PP^n}}^p|_X(d+p+1) \to \Omega_X^p(d+p+1),
\end{equation}
which are $\cG$-equivariant, since $i$ is an equivariant embedding. Since $X$ is cominuscule, $\Omega^1_X$ is completely reducible. Therefore, for any integers $p$ and $d$, the homogeneous vector bundle $\Omega_X^p(d+p+1)$ is also completely reducible. Then, due to Lemma \ref{lem:surjGeneral},  \eqref{display:restriction} induces a surjective map on global sections.  
\end{proof}

\begin{proof}[Proof of Theorem \ref{thm:folcomin}]
By definition, the embedding $i : X \to \PP(V_\lambda)$ is associated with the $\cG$-linearized line bundle $\cO_{\PP(V_\lambda)}(1)$ so that the restriction maps of twisted $p$-forms $i_p^*$ of \S \ref{subsection:restriction} are $\cG$-equivariant.
Recall from the setting \S \ref{subsection:restriction} that the subscheme $\IF(X,d+2)$ of integrable 1-forms on $X$ with values in $\cO_X(d+2)$ is defined in $\PP(H^0(\Omega^1_X(d+2))^\vee)$ by the equations determined by the map $\Psi_X$ sending one such form $\omega$ to $\Psi_X(\omega)=\omega \wedge d \omega$.
These yield a vector subspace quadratic equations in $S^2 H^0(\Omega^1_X(d+2))^\vee$.

Let us rewrite the commutative and diagram \eqref{comm_diag_rest} of $\cG$-equivariant linear maps:
\[
    \begin{tikzcd}
        S^2 H^0( \Omega^1_{{\PP^n}}(d+2)) \arrow[d, twoheadrightarrow, "\Psi_{{\PP^n}}" ] \arrow[r , "S^2i_1^*"]   & S^2 H^0(\Omega^1_X(d+2)) \arrow[d, "\Psi_X"] \\
        H^0( \Omega^3_{{\PP^n}}(2d+4))   \arrow[r,"i_3^*" ] & H^0(\Omega^3_X(2d+4)).
    \end{tikzcd}
\]
Since $X \subset \PP(V_\lambda)$ is a cominuscule variety,
Lemma \ref{lem:surjPbComin} applies and says that the horizontal arrows in the above diagram are surjective.
Then, by Lemma \ref{lem:surjideals2},
the linear subspace of $S^2 H^0(\Omega^1_X(d+2))^\vee$ consisting of equations of $\IF(X,d+2)$ 
is the image of the linear $\cG$-equivariant injective map $\Psi^\vee_X$.
\end{proof}

To finish this section we state a direct consequence of Proposition \ref{prop:restinject} in the case where $X$ is a cominuscule Grassmannian. The ample generator of the Picard group of $X$ is a very ample $\cG$-linearized line bundle $E_\lambda$ and we get an embedding of $X$ into $\PP(V_\lambda)$, called \textit{primitive embedding}. In particular, Proposition \ref{prop:restinject} applies and we get the following corollary.

\begin{Corollary}\label{cor:injproj}
Let $i\colon X \hookrightarrow \PP(V_\lambda)$ be the primitive embedding of a cominuscule Grassmannian. Then the projection map $\pi \colon \Fol(\PP(V_\lambda), 2) \to \Fol(X, 2)$ is a $\cG$-equivariant  embedding.
\end{Corollary}




\section{Distributions and Foliations on \texorpdfstring{$G(k,n)$}{G(k,n)}}\label{sect:grassmanians}

In this section we study the spaces $\IF(G(k,n),l)$ of integrable $1$-forms on a Grassmannian $G(k,n)$ of $k$-dimensional quotients of the vector space $V\cong \CC^n$. The Pl\"ucker embedding $i$ realizes $G(k,V)$ as a subvariety of $\PP(\extp^k V)$, in  a $\SL(V)$-equivariant manner. We are going to consider the induced maps 
\[
i_p^\ast \colon H^0(\Omega_{\PP(\extp^k V)}^p(d + p+1)) \longrightarrow H^0( \Omega_{G(k,V)}^p(d+p+1)),
\]
which are surjective due to Lemma \ref{lem:surjPbComin}. We will prove, in particular, that they induce isomorphisms between the spaces of codimension one degree zero foliations for $k=2$ and $n\geq 4$.

\subsection{Twisted forms on Grassmannians}

Let $\{e_1, \dots, e_n\}$ be the canonical basis of $V=\CC^n$ and denote by $l_i$ the weight of $e_i$ with respect to the natural action of $\fsl(V)$. Then $\alpha_{ij} = l_i - l_j$ are the roots and $\lambda_i = l_1 + \dots + l_i$ are the fundamental weights. 

Consider $G(k,V)$ the Grassmaniann of $k$-dimensional {quotients} of $V$. The spaces of twisted forms on $G(k,n)$ are described by the following result.

\begin{Lemma}
\label{lem:omdecgrass}
We have the following natural isomorphisms:
\begin{align*}
    H^0(\Omega^1_{G(k,V)}(d+2))  = &V_{\lambda_{k-1} + d\lambda_k + \lambda_{k+1}}, \\
    H^0(\Omega^3_{G(k,V)}(2d+4))  = &V_{3\lambda_{k-1} + 2d\lambda_k + \lambda_{k+3}} \oplus V_{ \lambda_{k-2}+\lambda_{k-1} + 2d\lambda_k + \lambda_{k+1}+\lambda_{k+2}} \\
    &\oplus V_{\lambda_{k-3} + 2d\lambda_k + 3\lambda_{k+1}}.
\end{align*}
\end{Lemma}

\begin{proof}
Let us write the tautological sequence
\[
0 \longrightarrow \U \longrightarrow V\otimes \OO_{G(k,V)} \longrightarrow \Q \longrightarrow 0
\]
where $\Q$ is the rank $k$ tautological quotient bundle and $\U$ is the rank $n-k$ tautological subbundle. In particular, the Plücker embedding is given by $\extp^k \Q = \OO_{G(k,V)}(1)$. 

As a homogeneous variety, the Grassmannian $G(k,V)$ is ${\rm SL}(n)/\cP_k$ , where $\cP_k$ is the maximal parabolic subgroup of ${\rm SL}(n)$ corresponding to the $k$-th fundamental root $\alpha_k$, cf. Example \ref{ex_sl_n}.
Then, homogeneous bundles correspond to representations of $\cP_k$ and irreducible ones correspond to representations of its semisimple part, hence to representations of $\SL(k)\times \SL(n-k)$, together with a twist by a line bundle. 

We also have that, according to \cite[p.60]{weyman:tract},
\[
\Omega^m_{G(k,V)} = \extp^m(\mathcal{Q}^\vee \otimes \mathcal{U}) = \bigoplus_{|\mu| = m}  \Gamma^\mu \mathcal{Q}^\vee \otimes \Gamma^{\mu'}\mathcal{U},
\]
where $\mu'$ is the dual partition of $\mu$. Then the lemma follows from  \cite[(2.3.3) Corollary]{weyman:tract}.
\end{proof}

\begin{Remark}
One can look at the homogeneous bundles appearing in the above proposition in the following pictorial way. The Grassmannian $G(k,n)$ corresponds to the $k$-root of $\mathrm{SL}(n)$ so we select the $k$-th vertex of a Dynkin diagram $A_{n-1}$ and the semisimple part of $\cP_k$ is associated to the diagram obtained by removing the $k$-th vertex. Note that this diagram is disconnected in general. A $\cP$-dominant weight $\lambda$ is of the form $\lambda=\sum_{i=1}^{n-1} a_i \lambda_i$ with $a_i \in \ZZ$ for all $i$ and $a_i\geq 0$ for $i \ne k$. We label the $i$-th vertex by $a_i$ to depict $E_\lambda$. We get:
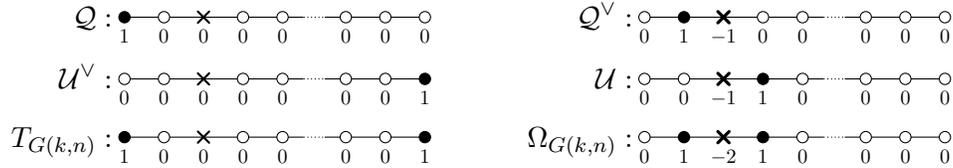
\begin{align*}
\newcommand\circleRoot[1]{\draw (root #1) circle (3pt);}
\mathcal{Q} : & \dynkin[%
labels={1,0,0,0,0,0,0,0},
label directions={,,,,,,,},
scale=1.5] A{*oxoo...ooo}
&
\mathcal{Q}^\vee : &\dynkin[%
labels={0,1,-1,0,0,0,0,0},
label directions={,,,,,,},
scale=1.5] A{o*Xoo...ooo} \\
 \mathcal{U}^\vee : &\dynkin[%
labels={0,0,0,0,0,0,0,1},
label directions={,,,,,,,},
scale=1.5] A{ooxoo...oo*}
&
\mathcal{U} : &\dynkin[%
labels={0,0,-1,1,0,0,0,0},
label directions={,,,,,,,},
scale=1.5] A{ooX*o...ooo} \\
T_{G(k,n)} : &\dynkin[%
labels={1,0,0,0,0,0,0,1},
label directions={,,,,,,,},
scale=1.5] A{*oxoo...oo*}
&
\Omega_{G(k,n)} : &\dynkin[%
labels={0,1,-2,1,0,0,0,0},
label directions={,,,,,,,},
scale=1.5] A{o*X*o...ooo} \\
\end{align*}

The diagram above is for $k=3$. Duality of a given vector bundle goes by dualizing representations of the factors $\SL(k)$ and $\SL(n-k)$ individually and adjusting the first Chern class by tensoring with $E_{a_k\lambda_k} = \cO_{G(k,n)}(a_k)$ for some $a_k \in \ZZ$. Here, $E_{\lambda_i}=\extp^i \mathcal{Q}$ for $0 \le i \le k$ while for $k+1 \le i \le n-1$ we have $E_{\lambda_i}=\extp^{n-i} \mathcal{U}^\vee$.

\end{Remark}

\subsection{Grassmannians of lines} Let us denote throughout this section by $X$ the Grassmannian $G(2,V)$. The following result shows that foliations of degree zero on $X=G(2,V)$ are precisely foliations of degree zero on the ambient Plücker space $\PP(V_{\lambda_2})$.

\begin{Theorem}\label{thm:isoG2n}
Let $V$ be a finite dimensional complex vector space. Then
\[
i_p^\ast \colon H^0(\Omega_{\PP(V_{\lambda_2})}^p(p+1)) \longrightarrow H^0( \Omega_{X}^p(p+1))
\]
is an isomorphism for $p =1, 3$. In particular we have the isomorphism
\[
\pi \colon G(2, V_{\lambda_2}^\vee) = \Fol(\PP(V_{\lambda_2}), 2) \stackrel{\sim}{\longrightarrow} \Fol(X, 2).
\]
\end{Theorem}

\begin{proof}
We observed in Lemma \ref{lem:surjPbComin} that the maps $i_1^*$ and $i_3^*$ are surjective.
We will prove that both source and target have the same decomposition as irreducible $\SL(V)$-modules, so that $i_1^*$ and $i_3^*$ are isomorphisms. Then the result follows from the third item of Lemma \ref{lem:surjideals2}, noting that $\IF(X,2) = \Fol(X,2)$. 

On the one hand, Lemma \ref{lem:omdecgrass} says that:
\[
 H^0(\Omega^1_{X}(2))= V_{\lambda_{1}  + \lambda_{3}} \quad \text{and} \quad      H^0(\Omega^3_{X}(4))  = V_{3\lambda_{1}  + \lambda_{5}} \oplus V_{ \lambda_{1}  + \lambda_{3}+\lambda_{4}} .
\]
On the other hand, it follows from \cite[Proposition 2.3.9]{weyman:tract}, which is a consequence of the Littlewood-Richardson rule, that 
\begin{align*}
    H^0(\Omega_{\PP(V_{\lambda_2})}^1(2)) & = \extp^2 \left(\extp^2 V\right) = \extp^2 V_{\lambda_2} = V_{\lambda_{1}  + \lambda_{3}}, \text{ and}\\
    H^0(\Omega_{\PP(V_{\lambda_2})}^3(4)) & = \extp^4 \left(\extp^2 V\right) = \extp^4 V_{\lambda_2} = V_{3\lambda_{1}  + \lambda_{5}} \oplus V_{ \lambda_{1}  + \lambda_{3}+\lambda_{4}}.
\end{align*}
Therefore $i_1^*$ and $i_3^*$ are isomorphisms. 

\end{proof}

\begin{Remark}
The case $n= \dim V = 5$ was already treated in \cite[Theorem 1.5, item (5.a)]{ACM:FanoDist}. Their proof relies on a previous result of Araujo and Druel about Mukai foliations and could not be generalized. One could also use our strategy to investigate distributions of class $k\geq 1$, see \cite{ACM:FanoDist} for a definition.
\end{Remark}

\subsection{Grassmannians of planes \texorpdfstring{$G(3,n)$}{G(3,n)}} In this section we deal with foliations under restriction for the Pl\"ucker embedding of $X = G(3,V)$. Due to Corollary \ref{cor:injproj} $\pi \colon  \Fol(\PP(V_{\lambda_3}), 2) \to \Fol(X, 2)$ is an embedding. The previous case of $G(2,V)$ suggests the problem of deciding whether $\pi$ is an isomorphism. In this direction we prove the following result.

\begin{Theorem}\label{thm:embG3n}
Let $V$ be a finite dimensional complex vector space and let $X = G(3,V)$. Consider the embedding
\[
\pi \colon \Fol(\PP(V_{\lambda_3}), 2) \hookrightarrow \Fol(X, 2) \subset \PP(H^0(\Omega_X^1(2))^\vee).
\]
Then  the ideals of  $Z:=\pi(\Fol(\PP(V_{\lambda_3}), 2))$ and $\Fol(X, 2)$, as subschemes of $\PP(H^0(\Omega_X^1(2))^\vee)$, agree in degree two, i.e.
\[
(I_{Z})_2 = (I_{\Fol(X, 2)})_2.
\]
\end{Theorem}

\begin{proof}
The proof boils down, owing to Lemma \ref{lem:surjideals2}, to proving that the square \eqref{comm_diag_dual} is cartesian in this case; this will be done explicitly by ad hoc computations.

From Lemma \ref{lem:omdecgrass} we know that 
\[
H^0(\Omega^1_{X}(2)) = V_{\lambda_{2} +  \lambda_{4}} \text{ and } H^0(\Omega^3_{X}(4)) = V_{3\lambda_{2}  + \lambda_{6}} \oplus V_{ \lambda_{1}+\lambda_{2}  + \lambda_{4}+\lambda_{5}} \oplus V_{ 3\lambda_{4}}.
\]
Moreover one can check (for instance with LiE \cite{lie}) that
\begin{align}\label{eq:decompPV31}
     H^0(\Omega_{\PP(V_{\lambda_3})}^1(2))  = \extp^2V_{\lambda_3}=& V_{\lambda_6} \oplus V_{\lambda_{2} +  \lambda_{4}} = V_{\lambda_6} \oplus H^0(\Omega^1_{X}(2)), \\ \label{eq:decompPV32}
     H^0(\Omega_{\PP(V_{\lambda_3})}^3(4))  = \extp^4V_{\lambda_3} =& V_{\lambda_{12}}\oplus V_{\lambda_3+\lambda_9 }\oplus V_{\lambda_2+\lambda_{10}}\oplus V_{2\lambda_6}   \oplus V_{\lambda_4+\lambda_8} \oplus \nonumber \\
    &\oplus   V_{2\lambda_2+\lambda_8}\oplus V_{\lambda_2+\lambda_3+\lambda_7}\oplus V_{\lambda_1+\lambda_4+\lambda_7}\oplus V_{\lambda_2+\lambda_4+\lambda_6} \oplus \\
    &\oplus H^0(\Omega^3_{G(k,V)}(4)). \nonumber
\end{align}
In particular, neither $i_1^*$ nor $i_3^*$ are injective. Furthermore we can compute
\begin{equation}\label{eq:decompS2V24}
    \begin{split}
S^2 V_{\lambda_2+\lambda_4}  = V_{2\lambda_6} \oplus V_{\lambda_4+\lambda_8} \oplus V_{3\lambda_4} \oplus V_{\lambda_3 +\lambda_4+\lambda_5} \oplus V_{\lambda_2+\lambda_4+\lambda_6}^{\oplus 2} \oplus V_{\lambda_2+\lambda_3+\lambda_7}  \oplus \\  \quad \oplus V_{2\lambda_2+\lambda_8}  \oplus V_{2\lambda_2+2\lambda_4}     \oplus V_{3\lambda_2+\lambda_6}  \oplus V_{\lambda_1+\lambda_5+\lambda_6}  \oplus V_{\lambda_1 +\lambda_4+\lambda_7} \oplus V_{\lambda_1+2\lambda_3+\lambda_5} \oplus  \\ \oplus V_{\lambda_1+ \lambda_2+\lambda_4+\lambda_5}  \oplus V_{\lambda_1+ \lambda_2 + \lambda_3+\lambda_6} \oplus V_{2\lambda_2+2\lambda_5}\oplus V_{2\lambda_1+\lambda_3+\lambda_7}
\end{split}
\end{equation}
and also
\begin{equation}\label{eq:decompS2V3}
\begin{split}
    S^2\extp^2 V_{\lambda_3} =  V_{\lambda_{12}} \oplus 
    V_{2\lambda_6}^{\oplus 2} \oplus
    V_{\lambda_4 + \lambda_8}^{\oplus 3}  \oplus
    V_{3\lambda_4}  \oplus
    V_{\lambda_3+\lambda_9}  \oplus  
    V_{\lambda_3+\lambda_4 + \lambda_5}  \oplus
    V_{\lambda_2+\lambda_{10}}^{\oplus 2}  \oplus \\ \oplus
    V_{\lambda_2+\lambda_4+\lambda_6}^{\oplus 3}  \oplus 
    V_{\lambda_2+\lambda_3+\lambda_7}^{\oplus 2}  \oplus
    V_{2\lambda_2+\lambda_8}^{\oplus 2}  \oplus
    V_{2\lambda_2+2\lambda_4}  \oplus
    V_{3\lambda_2+\lambda_6}  \oplus  
    V_{\lambda_1+\lambda_5+\lambda_6}  \oplus \\ \oplus
    V_{\lambda_1+\lambda_4+\lambda_7}^{\oplus 2}  \oplus
    V_{\lambda_1+\lambda_3+\lambda_8}  \oplus
    V_{\lambda_1+2\lambda_3+\lambda_5}  \oplus 
    V_{\lambda_1+\lambda_2+\lambda_9}  \oplus 
    V_{\lambda_1+\lambda_2+\lambda_4+\lambda_5}   \oplus \\ \oplus 
    V_{\lambda_1+\lambda_2+\lambda_3+\lambda_6}  \oplus 
    V_{2\lambda_1+2\lambda_5}  \oplus
    V_{2\lambda_1+\lambda_3+\lambda_7}.
\end{split}
\end{equation}

Let us write diagram \eqref{comm_diag_dual} in this case:
\begin{equation}
\label{comm_diag_dual_G(3,V)}
\begin{tikzcd}
S^2\extp^2 V_{\lambda_3}^\vee &  S^2 V_{\lambda_2+\lambda_4}^\vee \arrow[l, hookrightarrow, "\pi^*"'] \\
\extp^4 V_{\lambda_3}^\vee \arrow[u, hookrightarrow, "\Psi^\vee_{\PP(V_{\lambda_3})}"] & H^0(\Omega^3_{X}(4))^\vee \arrow[u, hookrightarrow, "\Psi^\vee_{X}"] \arrow[l, hookrightarrow] 
\end{tikzcd}.
\end{equation}
Note that $\pi^*$ comes from the decomposition  $\extp^2V_{\lambda_3} = V_{\lambda_6} \oplus V_{\lambda_{2} +  \lambda_{4}}$, hence it is just the inclusion of a direct factor. On the other hand, recall that from Remark \ref{rem:transtorep} the map $\Psi_{\PP(V_{\lambda_3})}$ is the multiplication map $(a\wedge b)\cdot (c\wedge d) \mapsto a\wedge b \wedge c \wedge d$, hence its dual is the diagonal (or comultiplication) map 
\begin{equation*}
    \Psi_{\PP(V_{\lambda_3})}^\vee (a\wedge b \wedge c \wedge d) = (a\wedge b) \cdot ( c \wedge d) - (a\wedge c) \cdot (b \wedge d)+ (a\wedge d) \cdot (b \wedge c),
\end{equation*}
see for instance \cite[Proposition 1.1.2]{weyman:tract}.

We want to show that 
\[
\pi^*(S^2 V_{\lambda_2+\lambda_4}^\vee) \cap \Psi^\vee_{\PP(V_{\lambda_3})}(\extp^4 V_{\lambda_3}^\vee) = \pi^* \circ \Psi^\vee_{X} (H^0(\Omega^3_{X}(4))^\vee).
\]
Since the maps involved are $\SL(V)$-equivariant, Schur's Lemma implies that we can work with each weight separately, i.e. we only need to prove that 
\[
\pi^*(V_{\mu}^\vee) \cap \Psi^\vee_{\PP(V_{\lambda_3})}(V_{\mu}^\vee) = \pi^* \circ \Psi^\vee_{X} (V_\mu^\vee)
\]
for each weight $\mu$.

First we deal with the weights ${3\lambda_{2}  + \lambda_{6}} , { \lambda_{1}+\lambda_{2}  + \lambda_{4}+\lambda_{5}}$ and ${ 3\lambda_{4}}$ appearing in $H^0(\Omega^3_{X}(4))^\vee$. We remark that they appear only once (there is only one irreducible direct factor for each weight) in each term of the square \eqref{comm_diag_dual_G(3,V)}. This follows from the decompositions \eqref{eq:decompPV31}, \eqref{eq:decompPV32}, \eqref{eq:decompS2V24} and \eqref{eq:decompS2V3}. Hence $\pi^*(V_{\mu}^\vee) \cap \Psi^\vee_{\PP(V_{\lambda_3})}(V_{\mu}^\vee)$ and $\pi^* \circ \Psi^\vee_{X} (V_\mu^\vee)$ can only be the unique $V_\mu^\vee$ appearing in $S^2\extp^2 V_{\lambda_3}^\vee$.

For a weight $\mu$ not appearing in $H^0(\Omega^3_{X}(4))^\vee$ we need to show that 
\begin{equation}\label{eq:nullintermu}
\pi^*(V_{\mu}^\vee) \cap \Psi^\vee_{\PP(V_{\lambda_3})}(V_{\mu}^\vee) = \{0\}.    
\end{equation}
For weights appearing only in $S^2 V_{\lambda_2+\lambda_4}^\vee$ or only in $\extp^4 V_{\lambda_3}^\vee$ we have nothing to do. We are reduced to analyse the weights that appear in both. From \eqref{eq:decompPV32} and \eqref{eq:decompS2V24} we see that common weights (not in  $\extp^4 V_{\lambda_3}^\vee$) are:
\[
{2\lambda_6} , \, {\lambda_4+\lambda_8} , \,{2\lambda_2+\lambda_8} , \, {\lambda_2+\lambda_3+\lambda_7}, \, {\lambda_1+\lambda_4+\lambda_7} \text{ and } {\lambda_2+\lambda_4+\lambda_6}.
\]
We claim that \eqref{eq:nullintermu} holds for each one of these weights, concluding the proof of the proposition. The proof of the claim will be given in Lemmas \ref{lem:comp2l6}, \ref{lem:compl2l4l6}, \ref{lem:compl4l8}, \ref{lem:comp2l2l8}, \ref{lem:compl2l3l7} and \ref{lem:compl1l4l7}.
\end{proof}

\begin{Remark}
Theorem \ref{thm:embG3n} raises the question whether it is true that $\Fol(\PP(V_{\lambda_3}) , 2) \to \Fol(X, 2)$ is an isomorphism for any $n$. For $n=6$ we show in Subsection \ref{sec_36} that $\Fol(\PP(V_{\lambda_3}) , 2) \cong \Fol(G(3,6), 2)_{red}$, where $X_{red}$ denotes the underlying reduced structure of $X$.
\end{Remark}

\section{Foliations on other cominuscule spaces}
\label{section:other cominuscule}

As described in Table \ref{tab:comin}, cominuscule Grassmannians consist of: classical Grassmannians, quadrics, spinor varieties $OG(n,2n)$, Lagrangian Grassmannians $IG(n,2n)$ and the exceptional varieties $E_6/\cP_1$ and $E_7/\cP_7$. In \S\ref{sect:grassmanians} foliations on Grassmannians were studied. For a quadric hypersurface $Q\subset \PP^n$ we have an isomorphism $\Fol(Q,2)\cong \Fol(\PP^n,2)$ (this is a particular case of \cite[Theorem 5.6]{ACM:FanoDist} which holds for general complete intersections). In this section we analyse the remaining cominuscule Grassmannians.

\subsection{Spinor varieties}
Throughout this section $X$ will denote $OG(n,2n)=\SO(2n)/\cP_n$, a spinor variety of type $D_n$. Let $V$ be a complex vector space of dimension $2n$ and consider a nondegenerate symmetric bilinear form $q$ on $V$. The variety parameterizing $n$-dimensional subspaces of $V$ that are isotropic with respect to $q$ has two irreducible components, each one isomorphic to $X$. The spinor variety is primitively embedded in $\PP(V_{\lambda_n})$. 

\begin{Theorem}
\label{thm_S_equations}
The space of codimension one minimal degree foliations $\Fol(X,2)\subset \PP(H^0(\Omega^1_X(2))^\vee)\cong \PP(V_{\lambda_{n-2}}^\vee)$ is defined by the quadratic equations given by the inclusion $$H^0(\Omega^3_X(4))^\vee=(V_{2\lambda_{n-3}}\oplus V_{\lambda_{n-4}+2\lambda_{n-1}})^\vee\subset S^2 V_{\lambda_{n-2}}^\vee.$$ 
Moreover, when $n=4,5$, these equations are exactly the equations of $\Fol(\PP(V_{\lambda_n}),2)\cong G(2,V_{\lambda_n}^\vee)$, thus identifying the two spaces of foliations.
\end{Theorem}

\begin{proof}
Let $\Q$ be the tautological quotient bundle of $G(n,2n)$. Then we can see $q$ as a global section of $S^2 \Q$ and as such $X$ is one connected component of the zero locus of $q$. From the normal sequence one can check that the tangent bundle $T_{X}$ is isomorphic to $\extp^2\Q_X$, where $\Q_X = {\Q}|_X$. 
Alternatively, the weights of $\Omega^1_X$ were described in \S\ref{sssect:cotangComin} and those of $\Omega^3_X$ can be computed the same way, hence one gets:
\begin{align*}
&\Omega^1_X(2)=E_{\lambda_{n-2}}, && H^0(\Omega^1_X(2))=V_{\lambda_{n-2}};\\
&\Omega^3_X(4)=E_{2\lambda_{n-3}}\oplus E_{\lambda_{n-4}+2\lambda_{n-1}}, &&
H^0(\Omega^3_X(4))=V_{2\lambda_{n-3}}\oplus V_{\lambda_{n-4}+2\lambda_{n-1}}.
\end{align*}

On the other hand, $H^0(\Omega^1_{\PP(V_{\lambda_{n}}^\vee)}(2)) = \extp^2 V_{\lambda_{n}}^\vee = \bigoplus_{j\geq 0} V_{\lambda_{n-2-4j}}^\vee$.  It follows that for $n = 4, 5$ we have $H^0(\Omega^1_{\PP(V_{\lambda_{n}}^\vee)}(2)) = H^0(\Omega^1_X(2))$. Similarly,  $H^0(\Omega^3_{\PP(V_{\lambda_{n}}^\vee)}(4)) = \extp^4 V_{\lambda_{n}}^\vee = H^0(\Omega^3_X(4))$ for $n=4,5$. This pattern breaks for $n\geq 6$. By applying Theorem \ref{thm:folcomin}, the result follows.

\end{proof}

\begin{Remark}
The last assertion is a new result only for $n=5$ because for $n=4$ the spinor variety $X$ is just a six-dimensional quadric.
\end{Remark}

\subsection{Lagrangian Grassmannians}
Now let $X:=IG(n,2n)=C_{n}/\cP_n$, the symplectic Grassmannian of maximal isotropic quotient spaces. This variety parameterizes $n$-dimensional quotient spaces of a $2n$-dimensional space which are isotropic with respect to a skew-symmetric non-degenerate bilinear form. 

\begin{Proposition}
\label{prop_SG_equations}
The space of codimension one minimal degree foliations $\Fol(X,2)\subset \PP(H^0(\Omega^1_{X}(2))^\vee)\cong \PP(V_{2\lambda_{n-1}}^\vee)$ is defined by the quadratic equations given by the inclusion $$H^0(\Omega^3_{X}(4))^\vee=(V_{3\lambda_{n-2}+\lambda_n}\oplus V_{\lambda_{n-3}+3\lambda_{n-1}})^\vee\subset S^2 V_{2\lambda_{n-1}}^\vee.$$ 
\end{Proposition}

\begin{proof}
Let $w$ denote the skew-symmetric form on $\CC^{2n}$.
Then $X$ can be seen as the zero locus of $w$, seen as a global section of $\extp^2 \Q$, inside $G(n,2n)$. From the normal sequence one can check that the tangent bundle $T_{X}$ is isomorphic to $S^2\Q_X$, where we denoted by $\Q_X$ the restriction to $X$ of the tautological quotient bundle on $G(n,2n)$. Alternatively, the weights of $\Omega^1_X$ were described in \S\ref{sssect:cotangComin} and those of $\Omega^3_X$ can be computed the same way, hence one gets:
\begin{align*}
&\Omega^1_{X}(2)=E_{2\lambda_{n-1}}, && H^0(\Omega^1_{X}(2))=V_{2\lambda_{n-1}}; \\
&\Omega^3_{X}(4)=E_{3\lambda_{n-2}+\lambda_n}\oplus E_{\lambda_{n-3}+3\lambda_{n-1}}, && H^0(\Omega^3_{X}(4))=V_{3\lambda_{n-2}+\lambda_n}\oplus V_{\lambda_{n-3}+3\lambda_{n-1}}.
\end{align*}
The variety $X$ is primitively embedded in $\PP(V_{\lambda_n})$. By Lemma \ref{lem:surjPbComin}, the morphisms $i_k^* \colon H^0(\PP(V_{\lambda_n}),\Omega^k_{\PP(V_{\lambda_n})}(k+1))\to H^0(\Omega^k_{X}(k+1))$ are surjective for $k=1,3$ and any $n\geq 3$. Note that these morphisms are never isomorphisms. By applying Lemma \ref{lem:surjideals2} we deduce the result.
\end{proof}

\subsection{The Cayley plane}
The Cayley plane is the $E_6$-homogeneous variety $X= E_6/\cP_1$. It can be described as the complex variety underlying the octonionic projective plane $\mathbb{OP}^2$, see \cite{landsberg-manivel:geometry}. We obtain the following theorem, affording a complete description of the space of codimension one minimal degree foliations $\Fol(X,2)$.

\begin{Theorem}
\label{thm:E6}
Let $X=E_6/\cP_1 \subset \PP(V_{\lambda_1})$ be the primitive embedding, then
\[
\Fol(X,2) \cong \Fol(\PP(V_{\lambda_1}),2) \cong G(2,V_{\lambda_1}^\vee).
\]
\end{Theorem}

\begin{proof}
From \S\ref{sssect:cotangComin} we know that $\Omega^1_{X}(2)=E_{\lambda_{3}}$ and similarly one can show that $\Omega^3_{X}(4)=E_{\lambda_{2}+\lambda_5}$. One can then check with LiE \cite{lie} that 
\begin{align*}
&    H^0(\Omega^1_{\PP(V_{\lambda_1})}(2))\cong \extp^2 V_{\lambda_1} \cong V_{\lambda_3}\cong H^0(\Omega^1_{X}(2)), \\
&H^0(\Omega^3_{\PP(V_{\lambda_1})}(4))\cong \extp^4 V_{\lambda_1} \cong V_{\lambda_2+\lambda_5}\cong H^0(\Omega^3_{X}(4)).
\end{align*}
We deduce from Lemma \ref{lem:surjPbComin} that the maps $i^*_1$ and $i^*_3$ are isomorphisms, concluding the proof.


\end{proof}

\subsection{Legendrian varieties from the Freudenthal magic square}
\label{sec_36}

Let $X$ be one of the following manifolds appearing in the third row of the Freudenthal magic square, see \cite{landsberg-manivel:freudenthal}: $IG(3,6)$, $G(3,6)$, $OG(6,12)$ or $E_7/\cP_7$. We will denote by $\Fol(X, 2)_{\mathrm{red}}$ the classical variety $\Fol(X, 2)\subset \PP(H^0(\Omega^1_{X}(2))^\vee)$ endowed with the reduced scheme structure. We prove the following result.

\begin{Theorem}\label{thm:fol36}
Let $X\subset \PP(V_\lambda)$ be the primitive embedding, then the projection 
\[
\pi \colon \Fol(\PP(V_{\lambda}) , 2) \longrightarrow 
\Fol(X, 2)_{\reddd}
\]
is an isomorphism. In particular $\Fol(X, 2)_{\reddd} \cong G(2, V_{\lambda}^\vee)$.
\end{Theorem}

In view of Corollary \ref{cor:injproj} we only need to show that $\pi$ is surjective. This can be achieved using that, for each $X$ as above, $H^0(\Omega^1_{\PP(V_\lambda)}(2)) = \extp^2 V_\lambda=\CC \oplus H^0(\Omega^1_X(2))$ where the trivial factor is spanned by a $G$-invariant contact form on $\PP(V_\lambda)$.

Let us recall the common construction of legendrian varieties from the third row of the Freudenthal magic square (we refer again to \cite{landsberg-manivel:freudenthal}). Let $Y$ be one among the following adjoint varieties $\cG'/\cP'$ for an exceptional group $\cG'$: $F_4/\cP_1$, $E_6/\cP_2$, $E_7/\cP_1$ or $E_8/\cP_8$. They appear in the fourth row of the Freudenthal magic square. From the general theory of homogeneous vector bundles, it follows that the fiber of the tangent bundle of $Y$ at a point stabilized by $\cP'$ is a $\cP'$-representation $W'$. Notice that the semisimple factor of $\cP'$ is the group $\cG$ for $X=G/P$ respectively equal to $IG(3,6)$, $G(3,6)$, $OG(6,12)$, $E_7/\cP_7$. Thus $\cG$ acts on $W'$ and it turns out that, as a $\cG$-representation, $W'=\CC\oplus W$ for a certain $\cG$-representation $W$. The variety $X$ is constructed as the minimal $\cG$-orbit in $\PP(W)$, and $W$ is the representation  $V_\lambda$ appearing in the embedding $X\subset \PP(V_\lambda)$. 

\begin{Lemma}
A form $w\in \extp^2 V_\lambda$ generating the trivial $G$-sub-representation is non-degenerate.
\end{Lemma}

\begin{proof}
Adjoint varieties are contact manifolds (see for instance \cite{BM_contact}) with contact structure given by $\theta \in H^0(\Omega_Y^1(1))$. Being a contact structure means 
that the induced distribution is regular:
\begin{equation*}
    0 \longrightarrow F \longrightarrow T_Y \stackrel{\theta}{\longrightarrow} \cO_Y(1) \longrightarrow 0,
\end{equation*}
and the $\cO_Y$-bilinear map $\extp^2 F \to \cO_Y(1)$ defined by $u \wedge v \mapsto d\theta(u,v) = \theta([u,v])$ is nondegenerate. 

From our previous discussion, taking fibers gives $F_x\otimes k(x) = V_\lambda$. On the other hand $\theta_x$ is $P'$-invariant, hence $G$-invariant. Since $\cG$ is semisimple, $d\theta_x$ defines a trivial one dimensional $\cG$-subrepresentation of $\extp^2 V_\lambda^\vee$. One can check (for instance with LiE \cite{lie}) that this trivial factor is unique, thus $w$ is also nondegenerate.

\end{proof}



Now we recall the commutative diagram 
\begin{equation}\label{eq:diagG36rest}
    \begin{tikzcd}
        H^0(\Omega_{\PP(V_{\lambda})}^1(2)) = \extp^2 V_{\lambda} \arrow[d, rightarrow, "\psi_{{\PP(V_\lambda)}}" ] \arrow[r , "i_1^*",twoheadrightarrow]   & H^0(\Omega^1_{X}(2))  \arrow[d, "\psi_{X}"] \\
        H^0(\Omega^3_{{\PP(V\lambda)}}(4)) = \extp^4 V_{\lambda}   \arrow[r,"i_3^*" , twoheadrightarrow] & H^0(\Omega^3_{X}(4)).
    \end{tikzcd}
\end{equation}
where $\psi_Y(v) = \Psi_Y(v \cdot  v)$ is the quadratic map associated to $\Psi_Y$, $Y$ being here $\PP(V_{\lambda})$ or $X$. Also recall that the cone over $\Fol(X,2)\subset \PP(H^0(\Omega_{X}^1(2))^\vee)$ is precisely the vanishing locus of $\psi_X$.

Next, as already noticed, one can check with LiE \cite{lie} that for all varieties in the third row of the Freudenthal magic square we have:
\begin{align*}
    H^0(\Omega_{\PP(V_{\lambda})}^1(2))  & = \extp^2 V_{\lambda} =  \mathbb{C} \oplus  H^0(\Omega^1_X(2)), \text{ and} \\
    H^0(\Omega_{\PP(V_{\lambda})}^3(4))  &= \extp^4V_{\lambda} = \extp^2 V_{\lambda} \oplus H^0(\Omega^3_{X}(4)),
\end{align*}
with $H^0(\Omega^1_X(2))$ being an irreducible representation. Let us denote by $w$ and $w'$ the highest weight vectors in $\extp^2 V_{\lambda}$, with $\CC w $ being the trivial representation. Now we are ready to proceed to the proof of Theorem \ref{thm:fol36}.

\begin{proof}[Proof of Theorem \ref{thm:fol36}]
As discussed above we only need to prove that $\pi$ is surjective. We start with $v \in \extp^2 V_{\lambda} $ such that $\psi_{X}\circ i_1^*(v) = 0$. Since $w$ generates $\ker i_1^*$, it is enough to show that there exists $a\in \CC$ and $u \in \extp^2 V_{\lambda}$ such that
\[
v = a\,w  + u, \text{ and } u\wedge u = 0.
\]
Recall that, owing to Remark \ref{rem:transtorep}, $\psi_{{\PP^n}} (v) = v\wedge v$. Using the commutativity of \eqref{eq:diagG36rest} we may assume that 
\[
 \psi_{{\PP^n}} (v) = v\wedge v \in \ker i_3^* = \CC\oplus H^0(\Omega^1_X(2)) \subset \extp^4 V_{\lambda}.
\]
And we claim that $v\wedge v \in \ker i_3^*$ implies that $w$ divides $v\wedge v$, i.e., there exists $x \in \extp^2 V_{\lambda}$ such that $v\wedge v = w\wedge x$. 

Let $\mathfrak{g}$ be the Lie algebra of $\cG$. To prove this claim we recall the universal enveloping algebra $U(\mathfrak{g}) = \bigoplus_{n\geq 0} \mathfrak{g} ^{\otimes n} / I$, where  $I$ is the ideal generated by $x\otimes y - y  \otimes x - [x,y]$ for every $x,y \in \mathfrak{g}$, and $\mathfrak{g} ^{\otimes 0} = \CC$. We can see the elements of $U(\mathfrak{g})$ as (non-commutative) polynomials on the elements of $\mathfrak{g}$. For more details we refer to \cite[Chapter III]{serre}. Any $\mathfrak{g}$-module affords an induced $U(\mathfrak{g})$-action and the property that we will use is the following. If $v_\mu \in V_\mu$ is the highest weight vector then $V_\mu = U(\mathfrak{g})\cdot v_\mu$, see \cite[Chapter VII]{serre}. 

Now note that $w\wedge w, w \wedge w'\in \extp^4 V_{\lambda}$ are highest weight vectors of weights $ 0$ and the highest weight of $H^0(\Omega^1_X(2))$, respectively. Then for $v\wedge v \in \mathbb{C} \oplus H^0(\Omega^1_X(2)) \subset \extp^4 V_{\lambda}$ there exist $P, Q \in U(\mathfrak{g})$ such that 
\[
v\wedge v = P\cdot (w\wedge w) + Q\cdot (w \wedge w').
\]
Note that $\mathfrak{g} \cdot w = \CC\, w$, hence $P\cdot w = p \, w$ and $Q\cdot w = q \, w$, for some $p,q\in \CC$. Developing the expression above we get   
\begin{align*}
  v\wedge v &=  2p\,(w\wedge w) + q\,(w \wedge w') + w\wedge  (Q\cdot w') \\ &= w \wedge \left( 2p \, w + (q + Q)\cdot w' \right).
\end{align*}
Then define $x:= 2p \, w + (q + Q)\cdot w'  \in \extp^2 V_{\lambda}$.

Next we claim that if $w$ divides $v\wedge v$ then there exists $a\in \CC $ and $u\in \extp^2 V_{\lambda} $ such that $v = a\,w  + u$ and $u\wedge u = 0$, concluding the proof of the theorem. The proof of this claim will be given in Lemma \ref{lem:2form36}.
\end{proof}

\begin{Lemma}\label{lem:2form36}
Let $v, w \in \extp^2 \CC^{2n}$ such that $w^{n} \neq 0$. Suppose that $w$ divides $v \wedge v$, i.e., there exists $u \in \extp^2 \CC^{2n}$ such that  $v \wedge v = w \wedge u$. Then there exist $a\in \CC$ and $y \in\extp^2 \CC^{2n} $ such that 
\[
v = aw + y \text{ and } y\wedge y = 0.
\]
\end{Lemma}

\begin{proof}
Given $z \in  \extp^k \mathbb{C}^{2n} $ define the set  $K_z^r :=  \{ x \in \extp^r \mathbb{C}^{2n} \mid x \wedge z =0 \}$. It follows from \cite[Theorem 1]{divgrass} that $w$ divides $v \wedge v$ if and only if $K_{w}^{2n-4} \subset K_{v \wedge v}^{2n-4}$.

Consider the pencil of $2$-forms $v + t w$. Since each bilinear alternating form corresponds to a skew-symmetric matrix, we get a pencil $A + tB$ of $2n \times 2n$ matrices. Recall that linear change of coordinates $P \in {\rm GL}(2n, \mathbb{C})$ on $\mathbb{C}^{2n}$ correspond to a congruence $P^T(A + tB)P$. 

Owing to \cite[Theorem 1]{T.pencils}, one has that $A + tB$ is, up to congruence, a block-diagonal matrix whose blocks depend on invariants of the pencil. Since $w^{n} \neq 0$ we have that $\det B \neq 0$, hence the only blocks that can appear have the form 
\[
B(a,m) := \begin{bmatrix}
0 & (a+t)\Delta_m + \Lambda_m \\  -(a+t)\Delta_m - \Lambda_m & 0
\end{bmatrix}_{2m\times 2m}
\]
where 
\[
\Delta_m = \begin{bmatrix}
0& & 1 \\
& \iddots & \\
1 && 0
\end{bmatrix}_{m\times m}, \quad \Lambda_m = 
\begin{bmatrix}
0 & & & 0 \\
& & 0& 1\\
&\iddots & \iddots& \\
0 & 1& & 0
\end{bmatrix}_{m\times m}.
\]
The only invariants are the elementary divisors $(a+t)^{2m}$ of $\det(A+tB)$. Fix an integer partition $\lambda_1\geq \lambda_2 \geq \dots \geq \lambda_r$, $\lambda_1 + \dots + \lambda_r = n$, and complex numbers $a_1, \dots, a_r\in \CC$ and define
\[
M(t) = (p_{i,j} + t q_{i,j}) := B(a_1, \lambda_1) \oplus \dots \oplus B(a_r,\lambda_r).
\]
Moreover, $v= \sum_{i<j}p_{i,j}e_i\wedge e_j$ and $w = \sum_{i<j}q_{i,j}e_i\wedge e_j$; note that $v$ depends on $a_1,\dots, a_r$ but $w$ does not. 

We note that if either $\lambda_1\geq 3$ or $\lambda_2\geq 2$ then $K_{w}^{2n-4} \not\subset K_{v \wedge v}^{2n-4}$. Indeed, if 
$\lambda_1 \geq 3$ then 
\[
w = \sum_{j=1}^{\lambda_1} e_{j}\wedge e_{2\lambda_1 +1- j} + w' \text{ and } v = a_1\sum_{j=1}^{\lambda_1} e_{j}\wedge e_{2\lambda_1 +1 - j} + \sum_{j=1}^{\lambda_1} e_{j+1}\wedge e_{2\lambda_1 +1 - j} +  v'
\]
where $w'$ and $v'$ only involve $e_j$ for $j\geq 2\lambda_1+1$. Then take $\phi = \bigwedge_{j \notin \{2,3,2\lambda_1-1, 2\lambda_1 \}} e_j$. It follows that $\phi \in K_{w}^{2n-4}$ but $\phi\wedge v\wedge v = \pm 1$. Similarly, if $\lambda_1 = \lambda_2 = 2$ we have
\begin{align*}
    w &= e_1\wedge e_4 + e_2\wedge e_3 + e_5\wedge e_8 + e_6\wedge e_7 + w' \text{ and } \\
    v &= a_1(e_1\wedge e_4 + e_2\wedge e_3) + a_2(e_5\wedge e_8 + e_6\wedge e_7) + e_2\wedge e_4 + e_6\wedge e_8 + v', 
\end{align*}
where $w'$ and $v'$ only involve $e_j$ for $j\geq 9$. Then $\phi = \bigwedge_{j \notin \{2,4,6,8 \}} e_j \in K_{w}^{2n-4} \setminus K_{v \wedge v}^{2n-4}$. Therefore we only need to deal with the partitions $(2,1^{n-2})$ and $(1^{n})$.

For the partition $(2,1^{n-2})$ we get
\begin{align*}
w &= e_1\wedge e_4 + e_2\wedge e_3 + \sum_{j=3}^{n} e_{2j-1}\wedge e_{2j} \text{ and } \\
v &= a_1(e_1\wedge e_4 + e_2\wedge e_3) +e_2\wedge e_4 + \sum_{j=3}^{n} a_{j-1}(e_{2j-1}\wedge e_{2j}). 
\end{align*}
Imposing $K_{w}^{2n-4} \subset K_{v \wedge v}^{2n-4}$ implies $a_1 = a_2 = \dots = a_{n-1}$, hence $v = a_1 w + e_2\wedge e_4$ and we are done.   

For the partition $(1^{n})$ we get 
\[
w = \sum_{j=1}^{n} e_{2j-1}\wedge e_{2j} \text{ and } v = \sum_{j=1}^{n} a_j(e_{2j-1}\wedge e_{2j}) 
\]
and imposing $K_{w}^{2n-4} \subset K_{v \wedge v}^{2n-4}$ we get that at least $n-1$ of the $n$ coefficients $a_i$ must be equal. Hence, up to reordering, we have $v = a_2w + (a_2-a_1)(e_1\wedge e_2)$ and we are done.
\end{proof}

\begin{Remark} It is still not clear to us whether $\Fol(X,2)$, seen as a scheme defined by the quadratic equations $H^0(\Omega^3_{X}(2))^\vee \subset S^2 H^0(\Omega^1_{X}(2))^\vee$ (see Theorem \ref{thm:folcomin}), is already a reduced scheme. We also remark that experimental computations with Macaulay2 \cite{M2-hyper} were crucial to discover the lemma above.
\end{Remark}


\section{Other Cases}
\label{section:other lines}

In this section we discuss some examples of homogeneous varieties $X\subset \PP$ other than primitively embedded cominuscule Grassmannians. We will see that the restriction $\pi\colon \Fol(\PP,2) \to \Fol(X,2)$ may still be an isomorphism for some cases while it fails for some others. 

\subsection{Symplectic Grassmannians of lines} We start with a positive result. Let $V$ be an even-dimensional vector space and let $\omega \in \extp^2 V$ be a symplectic form. For a subspace $W\subset V^\vee$ we have its orthogonal $W^{\perp} = \{\, v \in V^\vee \mid \forall w\in W,\, \omega(v,w) = 0  \,\}$. Moreover $W$ is called isotropic (with respect to $\omega$) if $W\subset W^\perp$. 
The variety $X = IG(2,V)$ parameterizing isotropic subspaces $W\subset V^\vee$ of dimension $2$ is called the symplectic Grassmannian of lines. It is a homogeneous variety $X = \Sp(V)/P_2$. 

Another way to describe $X$ is to consider the classical Grassmannian of lines $Y:=G(2,V)$, which carries $\Q$ the rank two tautological quotient bundle. Note that $\omega$ induces a global section of $\extp^2 \Q\cong \cO_Y(1)$ whose vanishing locus is precisely $X$. In particular, $X$ is a hyperplane section of $Y$ in its Plücker embedding. The space generated by $\omega$ defines a trivial $\Sp(V)$-subrepresentation of $\extp^2 V$, the resulting quotient is an irreducible $\Sp(V)$-module denoted by $\extp^{\langle 2 \rangle}V$. Then we may write $X = Y \cap \PP(\extp^{\langle 2 \rangle}V)$. 

\begin{Theorem} \label{thm:symplectic_lines}
Let $X=IG(2,V)$ and consider the embedding  $X \hookrightarrow \PP(\extp^{\langle 2 \rangle}V)$. Then the restriction map yields an isomorphism:
\[
\Fol(X,2) \cong \Fol(\PP(\extp^{\langle 2 \rangle}V),2) \cong G(2,\extp^{\langle 2 \rangle}V^\vee).
\]
\end{Theorem}

\begin{proof}
Ultimately we will prove that the restriction maps $i_p^*\colon H^0(\Omega_{\PP(\extp^{\langle 2 \rangle}V)}^p(p+1)) \to H^0(\Omega_{X}^p(p+1))$ are isomorphisms for $p=1,3$, so that the result follows from Lemma \ref{lem:surjideals2}.  

First we describe $\Omega_X^1$. Consider the dual tautological sequence of $Y = G(2,V)$
\[
0 \longrightarrow \Q^\vee \longrightarrow V^\vee\otimes\OO_Y \longrightarrow \U^\vee \longrightarrow 0.
\]
Then taking the fiberwise orthogonal subspace of $\Q^\vee\subset V^\vee\otimes\OO_Y$ with respect to $\omega\otimes 1$ defines a bundle $(\Q^\vee)^\perp$ on $Y$. From the exact sequence above we get $(\Q^\vee)^\perp \cong \U^{\vee\vee} = \U$. Let $\U_X := \U|_X$ and $\Q_X:=\Q|_X$ denote the restrictions to $X$. Since each point of $X$ corresponds to an isotropic subspace of $V^\vee$ we get an exact sequence
\[
0 \longrightarrow \Q^\vee_X \longrightarrow \U_X \longrightarrow \U_X/\Q^\vee_X \longrightarrow 0 
\]
Combining this sequence with the conormal sequence of $X\subset Y$ we arrive at the following diagram:
\[
\begin{tikzcd}
   & 0 \arrow[d] & 0\arrow[d] &  & \\
 & \OO_X(-1) \arrow[d] \arrow[r , equal ] &  \OO_X(-1)  \arrow[d]  &   & \\
  0 \arrow[r] & \Q_X^\vee\otimes \Q_X^\vee \arrow[d] \arrow[r ] & \U_X\otimes \Q_X^\vee \arrow[d] \arrow[r ] &  \left(\U_X/\Q^\vee_X \right)\otimes \Q_X^\vee\arrow[ d , equal] \arrow[ r] & 0 \\
0 \arrow[r]  & S^2\Q_X^\vee \arrow[d] \arrow[r] & \Omega_X^1 \arrow[r] \arrow[d]&\left(\U_X/\Q^\vee_X \right)\otimes \Q_X^\vee  \arrow[r]&0 \\
  & 0  & 0  &&
\end{tikzcd}
\]
The bottom row can be rewritten as 
\[
0 \longrightarrow E_{2\lambda_1 - 2\lambda_2} \longrightarrow \Omega_X^1 \longrightarrow E_{-2\lambda_2 +\lambda_3} \longrightarrow 0
\]
which cannot split since $X$ is not cominuscule. Applying Bott-Borel-Weil Theorem we get 
\[
H^0(\Omega_X^1(2)) = V_{2\lambda_1} \oplus  V_{\lambda_3} \cong  \extp^2  \extp^{\langle 2 \rangle}V = H^0(\Omega_{\PP(\extp^{\langle 2 \rangle}V)}^1(2)).
\]
Similarly one gets $H^0(\Omega_{\PP(\extp^{\langle 2 \rangle}V)}^3(4)) \cong H^0(\Omega_{X}^3(4))$.

Finally, we only need to show that the maps $i_p^*$ are surjective to conclude that they are isomorphisms. 
Consider the inclusions $X \subset \PP(\extp^{\langle 2 \rangle}V) \subset \PP(\extp^2V)$ and $X\subset Y \subset \PP(\extp^2V)$. Then we get a commutative diagram for $p=1,3$:
\[
 \begin{tikzcd}
   H^0(\Omega_{\PP(\extp^2V)}^p(p+1)) \arrow[r, twoheadrightarrow] \arrow[d, "\simeq"{    rotate=-90, anchor=south, inner sep = 1.5pt}] & H^0(\Omega_{\PP(\extp^{\langle 2 \rangle}V)}^p(p+1)) \arrow[d, "i_p^* "] \\
   H^0(\Omega_Y^p(p+1)) \arrow[r, "j_p^* "]  & H^0(\Omega_X^p(p+1))
\end{tikzcd}
\]
Then $i_p^*$ is surjective if and only if so is the corresponding $j_p^*$. The fact that $j_p^*$ is surjective for $p = 1,3$ can be verified directly. It boils down to the vanishing of $H^1(\Omega_X^{p-1}(p))$ which can be checked from the conormal sequence. 
\end{proof}

\begin{Remark}
When $\dim V = 2r+1$ one can still define $X = IG(2,V)$ the space of isotropic lines with respect to some $\omega \in \extp^2V$. It is still a hyperplane section of $G(2,V)$, smooth if $\omega$ maximal rank $2r$. However, in this case $X$ is not a homogeneous variety. Nonetheless the statement of Theorem \ref{thm:symplectic_lines} holds verbatim in this case with a modified proof based on the fact that $X\subset G(2,V)$ is a smooth hyperplane section. We plan to investigate foliations on hyperplane sections and on more general zero-loci in a follow-up work.

\end{Remark}




\subsection{Orthogonal Grassmannians of lines}
Now we have a negative result. Let $V$ be a vector space and let $q\in S^2V$ be a nondegenerate symmetric form. Then let $X = OG(2,V)$ be the space of dimension two subspaces of $V^\vee$ that are isotropic with respect to $q$. It is called the orthogonal Grassmannian of lines and it is a homogeneous variety for the group $\SO(V)$. Similar to the case of symplectic Grassmannians, $q$ induces a section of $S^2\Q$ over $G(2,V)$, whose vanishing locus is precisely $X$. Using the isotropy condition one also gets
\begin{equation}\label{seq:cotangOG}
    0 \longrightarrow \extp^2\Q_X^\vee = \OO_X(-1) \longrightarrow \Omega^1_X \longrightarrow \left(\U_X/\Q^\vee_X \right)\otimes \Q_X^\vee {= E_{\lambda_1 - 2\lambda_2 +\lambda_3 }} \longrightarrow 0
\end{equation}
In particular, $H^0(\Omega^1_X(1)) = \CC\zeta$ and this section $\zeta$ defines a contact distribution on $X$, thus $\zeta \wedge d\zeta^{\dim V-3}\neq 0$. Note that $X$ is an adjoint variety. 

\begin{Proposition}
\label{prop_OG_equations}
Consider $i\colon X = OG(2,V) \hookrightarrow \PP(V_{\lambda_2})$ the primitive embedding of the orthogonal Grassmannian of lines. Then $i^*_1 \colon H^0(\Omega^1_{\PP(V_{\lambda_2})}(2)) \to H^0(\Omega^1_{X}(2))$ is an isomorphism, but under this isomorphism we have a strict inclusion
\[
\Fol(\PP(V_{\lambda_2}),2) \subsetneq \Fol(X,2) \subset \PP(H^0(\Omega^1_{X}(2))^\vee).
\]
\end{Proposition}

\begin{proof}
First note that $\IF(X,2) = \Fol(X,2)$. Indeed, any element of $H^0(\Omega^1_{X}(2))$ with zeros in codimension one must be of the form $\omega = p\, \zeta$ where $p\in H^0(\OO_X(1))$ and $\zeta$ is the contact form. But $\omega \wedge d\omega = p^2 \, \zeta \wedge d\zeta \neq 0$.

Also note that the embedding $i$ factors through the Plücker embedding of $Y= G(2,V)$. Due to Theorem \ref{thm:isoG2n} we only need to consider $i\colon X \to Y$. In this setting we have the Koszul resolution
\begin{equation}\label{seq:koszulOG}
    0 \longrightarrow \OO_Y(-3) \longrightarrow S^2\Q^\vee(-1) \longrightarrow S^2\Q^\vee \longrightarrow \mathcal{I}_{X/Y} \longrightarrow 0
\end{equation}
and the commutative diagram 
\begin{equation}\label{diag:OG1}
    \begin{tikzcd}
   & 0 \arrow[d] & 0\arrow[d] &  & \\
 & \Omega^1_Y\otimes \mathcal{I}_{X/Y} \arrow[d] \arrow[r , equal ] &  \Omega^1_Y\otimes \mathcal{I}_{X/Y}  \arrow[d]  &   & \\
  0 \arrow[r] & K \arrow[d] \arrow[r ] & \Omega^1_Y\ \arrow[d] \arrow[r, "i_1^*" ] &  \Omega^1_X \arrow[ d , equal] \arrow[ r] & 0 \\
0 \arrow[r]  & S^2\Q_X^\vee \arrow[d] \arrow[r] & \Omega_Y^1|_X \arrow[r] \arrow[d]& \Omega^1_X  \arrow[r]&0 \\
  & 0  & 0  &&
\end{tikzcd}
\end{equation}
which are both $\SO(V)$-equivariant. Combining \eqref{diag:OG1} and \eqref{seq:koszulOG} and applying Bott-Borel-Weil Theorem we conclude that $H^0(K(2)) = H^1(K(2)) = 0$ which proves that $i^*_1$ is an isomorphism. Identifying $\Fol(\PP(V_{\lambda_2}),2)$ with its image we may say that 
\[
\Fol(\PP(V_{\lambda_2}),2) = \Fol(Y,2) \subset \Fol(X,2).
\]
We only need to prove that this inclusion is strict. Note that this is equivalent to show that $i_3^* \colon H^0( \Omega_Y^3(4)) \to H^0( \Omega_X^3(4))$ is not an isomorphism. 

On the one hand, taking exterior powers of \eqref{seq:cotangOG} and using Bott-Borel-Weil Theorem we get
\[
H^0(\Omega_X^3(4)) = V_{\lambda_2} \oplus V_{\lambda_1+\lambda_3} \oplus V_{2\lambda_2} \oplus V_{2\lambda_3} \oplus V_{2\lambda_1+\lambda_4} \oplus V_{\lambda_1+\lambda_2 + \lambda_3} \oplus V_{\lambda_1+\lambda_3 + \lambda_4} \oplus V_{3\lambda_1+\lambda_5}.
\]
On the other hand, as $\SO(V)$-modules,
\begin{align*}
    H^0(\Omega_Y^3(4)) = & \extp^4 V_{\lambda_2 } = V_{\lambda_2} \oplus V_{\lambda_1+\lambda_3} \oplus V_{2\lambda_2} \oplus V_{2\lambda_3} \oplus V_{2\lambda_1+\lambda_4} \oplus V_{\lambda_1+\lambda_2 + \lambda_3} \\ &\oplus V_{\lambda_1+\lambda_3 + \lambda_4} \oplus V_{3\lambda_1+\lambda_5} \oplus V_{\lambda_4} \oplus V_{\lambda_2+\lambda_4} \oplus V_{\lambda_1\oplus \lambda_5} \oplus V_{2\lambda_1}\oplus V_{2\lambda_1+\lambda_2}.
\end{align*}
It follows that $i_3^*$ cannot be injective, concluding the proof.
\end{proof}

\subsection{Some products of projective spaces}
A first example of a cominuscule variety of higher Picard rank is the product of projective spaces $X = \PP^{n_1}\times \dots \times \PP^{n_r}$. These varieties are primitively embedded, by the Segre embedding, in $\PP = \PP(H^0(\OO_X(1,1,\dots,1))$. For simplicity make $r=2$, then $X = \PP^{n_1} \times \PP^{n_2}$. It is easy to see that we cannot have $\Fol(X,2) \cong \Fol(\PP,2)$. For instance, consider the projection $\phi_1 \colon X \to \PP^{n_1}$ and let $\omega \in H^0 (\Omega_{\PP^{n_1}}^1(2))$ then $\pi_1^*\omega  \in H^0 (\Omega_{X}^1(2,0))$. If moreover $f\in H^0(\OO_{\PP^{n_2}}(2))$ then $\pi_2^*f \, \pi_1^*\omega  \in H^0 (\Omega_{X}^1(2,2))$ is an integrable form vanishing in codimension one, hence $\IF(X,2) \neq \Fol(X,2)$. Furthermore we can sow that $\IF(X,2)\not\cong \Fol(\PP,2)$.

\begin{Proposition}
\label{prop_PP_equations}
Let $i\colon X = \PP^{n_1}\times \PP^{n_2} \hookrightarrow \PP = \PP^{N}$ be the Segre embedding. Then $i_1^*$ is an isomorphism but $i_3^*$ is not injective. Therefore $\IF(X,2)$ is given by the equations
\[
H^0(\Omega^3_{X}(4,4))^\vee \subsetneq H^0(\Omega^3_{\PP}(4))^\vee \subset S^2 H^0(\Omega^1_{\PP}(2))^\vee.
\]
This gives a scheme theoretical strict inclusion $\Fol(\PP,2) \subsetneq \IF(X,2)$.
\end{Proposition}

\begin{proof}
Let $U$ and $V$ be vector spaces such that $X = \PP(U) \times \PP(V)$, hence $\PP = \PP(U\otimes V)$. Recall that the Segre embedding is $\SL(U)\times \SL(V)$-equivariant. Due to Lemma \ref{lem:surjPbComin} and Theorem \ref{thm:folcomin} the maps $i_p^*$ are surjective and $\IF(X,2)$ is given by the equations $H^0(\Omega^3_{X}(4,4))^\vee \subset S^2 H^0(\Omega^1_{\PP}(2))^\vee$. 

On the one hand, by Bott-Borel-Weil
\[
H^0(\Omega^1_{X}(2,2))=(\extp^2 U\otimes S^2V) \oplus (S^2U \otimes \extp^2 V)
\]
and, on the other hand, 
\[
H^0(\Omega^1_{\PP}(2)) =  \extp^2 (U\otimes V) = (\extp^2 U\otimes S^2V) \oplus (S^2U \otimes \extp^2 V)
\]
by the Littlewood-Richardson rule. Thus $i^*_1$ is an isomorphism. Similarly we compute 
\[
H^0(\Omega^3_{X}(4,4))=(\extp^4 U\otimes S^4V) \oplus (\Gamma^{3,1}U\otimes \Gamma^{2,1,1}V) \oplus (\Gamma^{2,1,1}U\otimes \Gamma^{3,1}V) \oplus (S^4U \otimes \extp^4 V).
\]
and $H^0(\Omega^3_{\PP}(4)) = H^0(\Omega^3_{X}(4,4)) \oplus (\Gamma^{2,2}U \otimes \Gamma^{2,2}V)$. Therefore $i^*_3$ is not injective, concluding the proof.
\end{proof}

\section{Further directions} \label{section:further}

Let us list some possible open questions and research directions that arise from this work. We will denote by $X$ a cominuscule Grassmannian. We could explicitly recover the space of foliations $\Fol(X,2)$ of some $X$ by showing that it is equal to the space of foliations of the ambient projective space (see Theorems \ref{thm:isoG2n}, \ref{thm_S_equations}, \ref{thm:E6}). A natural question is: 
\begin{Question}
Is $\overline{\pi(\Fol(\PP(V_\lambda),2))}$ isomorphic to $\Fol(X,2)$ for any cominuscule Grassmannian $X$? Or, equivalently, is the ideal of $\overline{\pi(\Fol(\PP(V_{\lambda}),2))}$ equal to the ideal of $\Fol(X,2)$?
\end{Question}
For Grassmannians $G(3,n)$ we could prove some results going in the direction of a positive answer to this question. These results concerned the ideal defining foliations and the properties of $\pi$ and give rise to the following questions.

\begin{Question}
Is the restriction of $\pi$ to $\Fol(\PP(V_{\lambda}),2)$ dominant over $\Fol(X,2)$?
\end{Question} 
We have already shown that this restriction is an embedding for cominuscule Grassmannians (Corollary \ref{cor:injproj}) and that it is a dominant embedding for the Legendrian varieties from the Freudenthal magic square (Theorem \ref{thm:fol36}).

\begin{Question}
Is $\Fol(X,2)$ smooth or at least reduced? 
\end{Question}
If we knew this was true, we would get from Theorem \ref{thm:fol36} that $\Fol(\PP(V_{\lambda}) , 2) \cong \Fol(X, 2)$ for $X$ one of the varieties appearing in the third row of the Freudenthal magic square.
\medskip

In some cases (see Theorems \ref{thm:folcomin}, \ref{thm_S_equations}, Propositions \ref{prop_OG_equations}, \ref{prop_SG_equations}
) we were able to provide the equations of the space of foliations $\Fol(X,2)$. 
These equations are often equivariantly unique, meaning that the irreducible representations they involve are unique in their respective ambient representation spaces. This naturally brings out the following problem:
\begin{Question}
Is it possible to recover $\Fol(X,2)$ as a $\cG$-variety solely from a description of its $\cG$-equivariant ideal?
\end{Question}
For instance, take the case of Grassmannians $G(3,n)$; the question (which is now a priori independent of the understanding of foliations on $G(3,n)$) is whether one can find a geometric description of the variety defined by the quadrics $$(V_{3\lambda_{4}}\oplus V_{3\lambda_2+\lambda_6 }\oplus V_{\lambda_1+\lambda_2+\lambda_4+\lambda_5})^\vee \subset S^2 V^\vee_{\lambda_2+\lambda_4}\cong H^0(\PP(V^\vee_{\lambda_2+\lambda_4}), \cO_{\PP(V^\vee_{\lambda_2+\lambda_4})}(2))$$ inside the projective space $\PP(V^\vee_{\lambda_2+\lambda_4})$. 
This kind of questions shows that there exists a very interesting bidirectional interplay between the study of foliations and the equivariant geometric theory.

\appendix

\section{Technical lemmas in the proof of Theorem \ref{thm:embG3n}}

We collected in this appendix some lemmas which are necessary to prove Thorem \ref{thm:embG3n}. From the square \eqref{comm_diag_dual_G(3,V)} we can derive the following diagram:
\begin{equation} \label{eq:diagG3n}
    \begin{tikzcd}
   V_{\lambda_6}^\vee \cdot \extp^2 V_{\lambda_3}^\vee &  S^2\extp^2 V_{\lambda_3}^\vee \arrow[l, twoheadrightarrow, "\xi"'] &  S^2 V_{\lambda_2+\lambda_4}^\vee \arrow[l, hookrightarrow, "\pi^*"'] \\
&\extp^4 V_{\lambda_3}^\vee \arrow[u, hookrightarrow, "\Psi^\vee_{\PP(V_{\lambda_3})}"] & 
     \end{tikzcd}.
\end{equation}
The top row comes from the second symmetric power of the decomposition $\extp^2 V_{\lambda_3}^\vee = V_{\lambda_6}^\vee \oplus V_{\lambda_2+\lambda_4}^\vee$, in particular it is exact. Then to show that \eqref{eq:nullintermu} holds for some $\mu$ it is enough to show, owing to Schur's Lemma, that $\xi \circ \Psi^\vee_{\PP(V_{\lambda_3})} (V^\vee_\mu) \neq \{0\}$.

In order to properly describe the map $\xi$ and do subsequent computations, let us fix some notation. Fix a basis $\{ e_1, \dots , e_n\}$ of $V^\vee$ such that each $e_i$ is a weight vector for the action of $\fsl(V^\vee)$. We write $e_{i_1, \dots , i_k }$ for the element $e_{i_1} \wedge \dots \wedge e_{i_k}\in V_{\lambda_k}^\vee = \extp^k V^\vee$. In particular, if $\sigma \in \mathcal{S}_k$ is a permutation,  
\[
e_{i_{\sigma(1)}, \dots , i_{\sigma(k)} } = (-1)^\sigma e_{i_1, \dots , i_k }
\]
and $e_{i_1, \dots , i_k } = 0 $ if $i_r = i_s$ for some pair of indices $r$ and $s$. Even though we only need $e_{i_1, \dots , i_k }$ with $i_1 < i_2 < \cdots < i_k$ in the definition of $V_{\lambda_k}^\vee$, it will be useful to consider all possible indices in view of the induced $\fsl(V^\vee)$ action. 

Let $l_i$ denote the weight of $e_i$, then each $e_{i_1, \dots, i_k }$ has weight $l_{i_1} + \dots + l_{i_k}$ for the induced $\fsl(V^\vee)$ action. If $\alpha_{r,s} = l_r-l_s$ is a root, then let $X_{r,s}\in \fsl(V^\vee)$ be the corresponding element: $X_{r,s}(e_t) = 0$ if $t\neq s$ and $X_{r,s}(e_s) = e_r$. Then the induced action gives 
\begin{align*}
    & X_{r,s}(e_{i_1, \dots , i_k })  = 0   \text{ if }  s\not\in I \text{ or } \{r,s\} \subset I, \\ 
    & X_{r,s}(e_{ \dots , s, \dots }) = e_{ \dots , r, \dots }  \text{ if }  s\in I \text{ and } r \not\in I
\end{align*}
where $I = \{ i_1 , \dots, i_k \}$. 

\begin{Remark}
\label{rem_basis}
Let us denote by $\prec$ the lexicographic order. A basis of $V_{\lambda_6}^\vee$ is given by $\{e_{i_1,\cdots,i_6}\}$ with $1\leq i_1<\cdots < i_6\leq n$, while a basis of $\extp^2 V^\vee_{\lambda_3}$ is given by $\{e_{i_1,i_2,i_3}\wedge e_{j_1,j_2,j_3}\}$ with $1\leq i_1<i_2 < i_3\leq n$, $1\leq j_1<j_2 < j_3\leq n$ and $(i_1,i_2 ,i_3)\prec (j_1,j_2 , j_3)$. Notice that the elements of the latter basis for which $\{i_1,i_2 ,i_3\}\cap \{j_1,j_2 , j_3\}\neq 0$ belong to $V_{\lambda_2+\lambda_4}^\vee$. As a consequence of this, the elements $\{e_{i_1,i_2,i_3}\wedge e_{j_1,j_2,j_3} \cdot e_{k_1,\cdots,k_6}\}$ with the previous conditions are a generating set of $\extp^2 V^\vee_{\lambda_3}\cdot V^\vee_{\lambda_6}$, and those such that $\{i_1,i_2 ,i_3\}\cap \{j_1,j_2 , j_3\}\neq 0$ are linearly independent in $V_{\lambda_2+\lambda_4}^\vee\otimes V^\vee_{\lambda_6}\subset \extp^2 V^\vee_{\lambda_3}\cdot V^\vee_{\lambda_6}$.\end{Remark}

\begin{Lemma}\label{lem:xi}
The map $\xi \colon S^2\extp^2 V_{\lambda_3}^\vee  \to V_{\lambda_6}^\vee \cdot \extp^2 V_{\lambda_3}^\vee$ from \eqref{eq:diagG3n} is defined by 
\[
\xi((e_{i,j,k}\wedge e_{l,m,n})\cdot (e_{o,p,q} \wedge e_{r,s,t})) = e_{i,j,k,l,m,n}\cdot (e_{o,p,q} \wedge e_{r,s,t}) + e_{o,p,q,r,s,t} \cdot (e_{i,j,k}\wedge e_{l,m,n})
\]
\end{Lemma}

\begin{proof}
First we note that the unique (up to scalar multiple) $\SL(V^\vee)$-equivariant map $\extp^2 V_{\lambda_3}^\vee \to V_{\lambda_6}^\vee$ is the multiplication $m(e_{i_1, i_2,i_3} \wedge e_{j_1, j_2,j_3}) = e_{i_1, i_2,i_3,j_1,j_2,j_3}$. Then $\xi$ can be described as the composition 
\[
\xi \colon S^2\extp^2 V_{\lambda_3}^\vee  \stackrel{\delta}{\longrightarrow} \extp^2 V_{\lambda_3}^\vee \otimes \extp^2 V_{\lambda_3}^\vee \stackrel{m\otimes 1}{\longrightarrow} V_{\lambda_6}^\vee \cdot \extp^2 V_{\lambda_3}^\vee
\]
where $\delta(u\cdot v) = (u\otimes v + v \otimes u)$. Therefore
\begin{align*}
    \xi( (a\wedge b) \cdot (c\wedge d)) & = (m\otimes 1)( (a\wedge b)  \otimes (c\wedge d) + (c\wedge d)\otimes (a\wedge b)  ) \\
    & = m(a\wedge b) \cdot (c\wedge d) + m(c\wedge d)\cdot (a\wedge b)  
\end{align*}
and
\[
\xi((e_{i,j,k}\wedge e_{l,m,n})\cdot (e_{o,p,q} \wedge e_{r,s,t})) = e_{i,j,k,l,m,n}\cdot (e_{o,p,q} \wedge e_{r,s,t}) + e_{o,p,q,r,s,t} \cdot (e_{i,j,k}\wedge e_{l,m,n}).
\]
\end{proof}

\begin{Remark}\label{rem:cyclicW2V3}
For further computations it will be useful to have the highest weight vectors of $\extp^2V_{\lambda_3}^\vee = V_{\lambda_6}^\vee \oplus V_{\lambda_{2} +  \lambda_{4}}^\vee$. Firstly notice that 
\[
w_{2,4} := e_{1,2,3}\wedge e_{1,2,4} 
\]
is the unique vector in $\extp^2V_{\lambda_3}^\vee$ of weight $\lambda_{2} +  \lambda_{4}$, hence it must be cyclic. For $\lambda_6$ we may use the diagonal map $\delta \colon V_{\lambda_6}^\vee \hookrightarrow \extp^2V_{\lambda_3}^\vee$ from which we define
\[
w_{6} := \delta(e_{1,2,3,4,5,6}) = \sum_{\sigma \in \mathcal{S}_6} (-1)^\sigma e_{\sigma_1,\sigma_2, \sigma_3}\wedge e_{\sigma_4,\sigma_5, \sigma_6}
\]
where $(-1)^\sigma$ denotes the sign of the permutation $\sigma$. Also note that $$
w_6^{10} = \underbrace{ w_6 \wedge \dots \wedge w_6}_{10 \text{ times}} \neq 0
$$ 
but $w_6^{11} = 0$.
\end{Remark}

Now we are ready to prove the technical bulk of Theorem \ref{thm:embG3n}.

\begin{Lemma}\label{lem:comp2l6}
The equality \eqref{eq:nullintermu} holds for $\mu = 2\lambda_6$.
\end{Lemma}

\begin{proof}
From Remark \ref{rem:cyclicW2V3} we have that $w_6 \in \extp^2V_{\lambda_3}^\vee $ is a highest weight vector of weight $\lambda_6$. Then 
\[
w_6\wedge w_6 = \sum_{\sigma, \tau \in \mathcal{S}_6} (-1)^\sigma (-1)^\tau e_{\sigma_1,\sigma_2, \sigma_3}\wedge e_{\sigma_4,\sigma_5, \sigma_6} \wedge e_{\tau_1,\tau_2, \tau_3}\wedge e_{\tau_4,\tau_5, \tau_6} \in \extp^4V_{\lambda_3}^\vee
\]
is also a highest weight vector of weight $2\lambda_6$. We then compute 
\[
\begin{split}
    \Psi^\vee_{\PP(V_{\lambda_3})}(w_6\wedge w_6) = \sum_{\sigma, \tau \in \mathcal{S}_6} (-1)^\sigma (-1)^\tau [(e_{\sigma_1,\sigma_2, \sigma_3}\wedge e_{\sigma_4,\sigma_5, \sigma_6})\cdot (e_{\tau_1,\tau_2, \tau_3}\wedge e_{\tau_4,\tau_5, \tau_6})  \\
    - (e_{\sigma_1,\sigma_2, \sigma_3}\wedge  e_{\tau_1,\tau_2, \tau_3}) \cdot (e_{\sigma_4,\sigma_5, \sigma_6} \wedge e_{\tau_4,\tau_5, \tau_6} )
    + (e_{\sigma_1,\sigma_2, \sigma_3}\wedge e_{\tau_4,\tau_5, \tau_6}) \cdot (e_{\sigma_4,\sigma_5, \sigma_6} \wedge e_{\tau_1,\tau_2, \tau_3}) ].
\end{split}
\]
Observe that 
\[
\begin{split}
    \xi((e_{\sigma_1,\sigma_2, \sigma_3}\wedge e_{\sigma_4,\sigma_5, \sigma_6})\cdot (e_{\tau_1,\tau_2, \tau_3}\wedge e_{\tau_4,\tau_5, \tau_6})) = (-1)^\sigma e_{1,2,3,4,5,6} \cdot (e_{\tau_1,\tau_2, \tau_3}\wedge e_{\tau_4,\tau_5, \tau_6}) \\+ (-1)^\tau e_{1,2,3,4,5,6} \cdot (e_{\sigma_1,\sigma_2, \sigma_3}\wedge e_{\sigma_4,\sigma_5, \sigma_6})
\end{split}
\]
and also notice that $\xi( (e_{\sigma_1,\sigma_2, \sigma_3}\wedge  e_{\tau_1,\tau_2, \tau_3}) \cdot (e_{\sigma_4,\sigma_5, \sigma_6} \wedge e_{\tau_4,\tau_5, \tau_6} )) \neq 0 $ if and only if $\{\tau_1,\tau_2, \tau_3\} = \{\sigma_4,\sigma_5, \sigma_6\}$. In this case there exist (unique) $\rho_1 , \rho_2  \in \mathcal{S}_3$ such that $\rho_1 (\tau_1,\tau_2, \tau_3) = (\sigma_4,\sigma_5, \sigma_6)$ and $\rho_2 (\tau_4,\tau_5, \tau_6) = (\sigma_1,\sigma_2, \sigma_3)$ hence
\[
\begin{split}
    \xi( (e_{\sigma_1,\sigma_2, \sigma_3}\wedge  e_{\tau_1,\tau_2, \tau_3}) \cdot (e_{\sigma_4,\sigma_5, \sigma_6} \wedge e_{\tau_4,\tau_5, \tau_6} )) = \quad \quad\quad\quad\quad \quad \quad\quad\quad\quad \\ = (-1)^{\rho_1}(-1)^{\rho_2} \xi( (e_{\sigma_1,\sigma_2, \sigma_3}\wedge  e_{\sigma_4,\sigma_5, \sigma_6}) \cdot (e_{\sigma_4,\sigma_5, \sigma_6} \wedge e_{\sigma_1,\sigma_2, \sigma_3} )) = \\
    = -2(-1)^{\rho_1}(-1)^{\rho_2} (-1)^{\sigma} e_{1,2,3,4,5,6}\cdot ( e_{\sigma_1,\sigma_2, \sigma_3}\wedge e_{\sigma_4,\sigma_5, \sigma_6} )= \\
    = 2(-1)^{\tau} e_{1,2,3,4,5,6}\cdot ( e_{\sigma_1,\sigma_2, \sigma_3}\wedge e_{\sigma_4,\sigma_5, \sigma_6} )
\end{split}
\]

The last equality comes from the fact that the joint (or concatenated) permutation $(\rho_1\mid \rho_2)\in \mathcal{S}_6$ satisfies $(\rho_1\mid \rho_2) \tau = (\sigma_4,\sigma_5, \sigma_6, \sigma_1,\sigma_2, \sigma_3)$ hence $(-1)^{(\rho_1\mid \rho_2)} = (-1)^{\rho_1}(-1)^{\rho_2} = - (-1)^{\tau}(-1)^\sigma$. Analogously, we have 
\[
\begin{split}
    \xi( (e_{\sigma_1,\sigma_2, \sigma_3}\wedge e_{\tau_4,\tau_5, \tau_6}) \cdot (e_{\sigma_4,\sigma_5, \sigma_6} \wedge e_{\tau_1,\tau_2, \tau_3})) = \quad \quad\quad\quad\quad \quad \quad\quad\quad\quad \\ = -2(-1)^{\rho_1'}(-1)^{\rho_2'} (-1)^{\sigma} e_{1,2,3,4,5,6}\cdot ( e_{\sigma_1,\sigma_2, \sigma_3}\wedge e_{\sigma_4,\sigma_5, \sigma_6} ) = \\
    = -2(-1)^{\tau} e_{1,2,3,4,5,6}\cdot ( e_{\sigma_1,\sigma_2, \sigma_3}\wedge e_{\sigma_4,\sigma_5, \sigma_6} )
\end{split}
\]
for (unique) $\rho_1', \rho_2'$ such that $\rho'_1(\tau_4,\tau_5, \tau_6) = (\sigma_4,\sigma_5, \sigma_6)$ and $\rho_2' (\tau_1,\tau_2, \tau_3) =(\sigma_1,\sigma_2, \sigma_3) $, hence  $(-1)^{\rho_1'}(-1)^{\rho_2'} = (-1)^{\tau}(-1)^\sigma$. Therefore 
\[
\xi \circ \Psi^\vee_{\PP(V_{\lambda_3})}(w_6\wedge w_6)  = A - B +C
\]
where 
\begin{align*}
    A & = \sum_{\sigma,\tau \in \mathcal{S}_6}  (-1)^\sigma (-1)^\tau [  (-1)^\sigma e_{1,2,3,4,5,6} \cdot (e_{\tau_1,\tau_2, \tau_3}\wedge e_{\tau_4,\tau_5, \tau_6}) + \\ & \quad \quad + (-1)^\tau e_{1,2,3,4,5,6} \cdot (e_{\sigma_1,\sigma_2, \sigma_3}\wedge e_{\sigma_4,\sigma_5, \sigma_6})]= \\
    & = 6!\,2 \, e_{1,2,3,4,5,6} \cdot w_6 = 1440\, e_{1,2,3,4,5,6} \cdot w_6,  \\
    B & =  \sum_{\sigma \in \mathcal{S}_6} \sum_{\stackrel{\tau \in  \mathcal{S}_6}{\{\tau_1,\tau_2, \tau_3\} = \{\sigma_4,\sigma_5, \sigma_6\}}}  (-1)^\sigma (-1)^\tau 2(-1)^{\tau} e_{1,2,3,4,5,6}\cdot ( e_{\sigma_1,\sigma_2, \sigma_3}\wedge e_{\sigma_4,\sigma_5, \sigma_6} ) = \\
    &= 72 \, e_{1,2,3,4,5,6} \cdot w_6 , \\
    C & = \sum_{\sigma \in \mathcal{S}_6} \sum_{\stackrel{\tau \in  \mathcal{S}_6}{\{\tau_1,\tau_2, \tau_3\} = \{\sigma_1,\sigma_2, \sigma_3\}}}  -(-1)^\sigma (-1)^\tau 2(-1)^{\tau} e_{1,2,3,4,5,6}\cdot ( e_{\sigma_1,\sigma_2, \sigma_3}\wedge e_{\sigma_4,\sigma_5, \sigma_6} ) = \\
    &= -72 \, e_{1,2,3,4,5,6} \cdot w_6.
\end{align*}
Therefore $\xi \circ \Psi^\vee_{\PP(V_{\lambda_3})}(w_6\wedge w_6)  = 1296 \, e_{1,2,3,4,5,6} \cdot w_6 \neq 0$.
\end{proof}

\begin{Lemma}\label{lem:compl2l4l6}
The equality \eqref{eq:nullintermu} holds for $\mu = \lambda_2+\lambda_4+\lambda_6$.
\end{Lemma}

\begin{proof}
As in the previous lemma, we get from Remark \ref{rem:cyclicW2V3} that the highest weight vector in $\extp^6V_{\lambda_3}^\vee$ of weight $\lambda_2+\lambda_4+\lambda_6$ is $w_{2,4}\wedge w_6$. Then 
\begin{align*}
    \Psi^\vee_{\PP(V_{\lambda_3})}(w_{2,4}\wedge w_6)  & = \Psi^\vee_{\PP(V_{\lambda_3})}(e_{1,2,3}\wedge e_{1,2,4} \wedge \sum_{\sigma \in \mathcal{S}_6} (-1)^\sigma e_{\sigma_1,\sigma_2, \sigma_3}\wedge e_{\sigma_4,\sigma_5, \sigma_6}) = \\
    & = (e_{1,2,3}\wedge e_{1,2,4} ) \cdot  \sum_{\sigma \in \mathcal{S}_6} (-1)^\sigma e_{\sigma_1,\sigma_2, \sigma_3}\wedge e_{\sigma_4,\sigma_5, \sigma_6}) + \\
    &\quad - \sum_{\sigma \in \mathcal{S}_6} (-1)^\sigma (e_{1,2,3}\wedge e_{\sigma_1,\sigma_2, \sigma_3}) \cdot (e_{1,2,4} \wedge e_{\sigma_4,\sigma_5, \sigma_6}) + \\
    &\quad + \sum_{\sigma \in \mathcal{S}_6} (-1)^\sigma (e_{1,2,3}\wedge e_{\sigma_4,\sigma_5, \sigma_6}) \cdot (e_{1,2,4} \wedge e_{\sigma_1,\sigma_2, \sigma_3 }). 
\end{align*}
Notice that $m(e_{1,2,3}\wedge e_{1,2,4}) = e_{1,2,3,1,2,4} = 0$. By the same reason we also have that $\xi((e_{1,2,3}\wedge e_{i,j,k}) \cdot (e_{1,2,4} \wedge e_{l,m,n})) \neq 0$ if and only if either $\{i,j,k \} = \{4,5,6\}$ or $\{l,m,n\} = \{3,5,6\}$. Following the same strategy of the proof of Lemma \ref{lem:comp2l6} we arrive at
\[
\xi \circ \Psi^\vee_{\PP(V_{\lambda_3})}(w_{2,4}\wedge w_6) = 576\, e_{1,2,3,4,5,6} \cdot w_{2,4} \neq 0.
\]
\end{proof}

\begin{Lemma}\label{lem:compl4l8}
The equality \eqref{eq:nullintermu} holds for $\mu = \lambda_4+\lambda_8$.
\end{Lemma}

\begin{proof}
Consider $V_{\lambda_4+\lambda_8}^\vee \subset V_{\lambda_4}^\vee\otimes V_{\lambda_8}^\vee$ generated by the highest weight vector $e_{1,2,3,4} \otimes e_{1,\dots,8}$. Now consider the diagonal maps:
\begin{align*}
    V_{\lambda_4}^\vee \hookrightarrow (V^\vee)^{\otimes 4}  & : \quad e_{1,2,3,4} \longmapsto \sum_{\sigma \in \mathcal{S}_4} (-1)^\sigma e_{\sigma_1} \otimes e_{\sigma_2} \otimes  e_{\sigma_3} \otimes e_{\sigma_4};  \\
    V_{\lambda_8}^\vee \hookrightarrow (V^\vee_{\lambda_2})^{\otimes 4} & : \quad e_{1,\dots,8} \longmapsto \sum_{\tau \in \mathcal{S}_8} (-1)^\tau e_{\tau_1 , \tau_2 } \otimes e_{\tau_3 , \tau_4 } \otimes e_{\tau_5 , \tau_6 } \otimes e_{\tau_7 , \tau_8 }.
\end{align*}
After applying the multiplication maps $V^\vee \otimes  V^\vee_{\lambda_2} \to V^\vee_{\lambda_3}$ and $(V^\vee_{\lambda_3})^{\otimes 4} \to \extp^4 V^\vee_{\lambda_3}$ we get a copy of $V_{\lambda_4+\lambda_8}^\vee$ inside $\extp^4 V^\vee_{\lambda_3}$ determined by 
\[
w_{4,8} =  \sum_{\sigma \in \mathcal{S}_4}\sum_{\tau \in \mathcal{S}_8} (-1)^\sigma(-1)^\tau e_{\sigma_1 ,\tau_1 , \tau_2 } \wedge e_{\sigma_2, \tau_3 , \tau_4 } \wedge e_{\sigma_3, \tau_5 , \tau_6 } \wedge e_{\sigma_4, \tau_7 , \tau_8 }.
\]

Let us show that $\xi \circ \Psi^\vee_{\PP(V_{\lambda_3})} (w_{4,8}) \neq 0$. First notice that, by symmetry,
$$
\Psi^\vee_{\PP(V_{\lambda_3})} (w_{4,8})=\sum_{\sigma \in \mathcal{S}_4}\sum_{\tau \in \mathcal{S}_8} (-1)^\sigma(-1)^\tau e_{\sigma_1 ,\tau_1 , \tau_2 } \wedge e_{\sigma_2, \tau_3 , \tau_4 } \cdot e_{\sigma_3, \tau_5 , \tau_6 } \wedge e_{\sigma_4, \tau_7 , \tau_8 }
$$
and
$$
\xi\circ \Psi^\vee_{\PP(V_{\lambda_3})} (w_{4,8})=6\sum_{\sigma \in \mathcal{S}_4}\sum_{\tau \in \mathcal{S}_8} (-1)^\sigma(-1)^\tau e_{\sigma_1 ,\tau_1 , \tau_2 } \wedge e_{\sigma_2, \tau_3 , \tau_4 } \cdot e_{\sigma_3, \sigma_4, \tau_5 , \tau_6 , \tau_7 , \tau_8 }.
$$
In order to show that this element is different from zero, let us show that one of its coefficients in the basis described in Remark \ref{rem_basis} is nonzero. More precisely, let us compute its coefficient with respect to the element $e_{1,2,3}\wedge e_{2,1,4}\cdot e_{3,4,5,6,7,8}$. To compute this coefficient, we need to isolate the permutations $\sigma$ and $\tau$ such that $e_{\sigma_1 ,\tau_1 , \tau_2 } \wedge e_{\sigma_2, \tau_3 , \tau_4 } \cdot e_{\sigma_3, \sigma_4, \tau_5 , \tau_6 , \tau_7 , \tau_8 }=\pm e_{1,2,3}\wedge e_{2,1,4}\cdot e_{3,4,5,6,7,8}$. In order for this to happen, $\sigma$ must send $\{1,2\}$ to $\{1,2\}$ (there are two possibilities since we need to take into account the order) and $\{3,4\}$ to $\{3,4\}$ (2 possibilities). If $\sigma_1=1$ then either $\tau$ must send $\{1,2\}$ to $\{2,3\}$ ($2$ possibilities), $\{3,4\}$ to $\{1,4\}$ ($2$ possibilities) and $\{5,6,7,8\}$ to $\{5,6,7,8\}$ ($4!$ possibilities) or it must send $\{1,2\}$ to $\{2,4\}$ ($2$ possibilities), $\{3,4\}$ to $\{1,3\}$ ($2$ possibilities) and $\{5,6,7,8\}$ to $\{5,6,7,8\}$ ($4!$ possibilities); similarly if $\sigma_1=2$. Since all these terms come with a $+$ sign, we obtain that the coefficient in question is equal to $2^83\neq 0$.
\end{proof}

\begin{Lemma}\label{lem:comp2l2l8}
The equality \eqref{eq:nullintermu} holds for $\mu = 2\lambda_2+\lambda_8$.
\end{Lemma}

\begin{proof}
Consider $V_{2\lambda_2 + \lambda_8}^\vee \subset V_{\lambda_2}^\vee\otimes V_{\lambda_2}^\vee\otimes V_{\lambda_8}^\vee$ determined by $e_{1,2}\otimes e_{1,2} \otimes e_{1, \dots,8}$. Then consider the diagonal map
\[
V_{\lambda_8}^\vee \hookrightarrow V^\vee \otimes V^\vee \otimes V_{\lambda_3}^\vee \otimes V_{\lambda_3}^\vee : \quad e_{1, \dots,8} \longmapsto \sum_{\sigma \in \mathcal{S}_8} (-1)^\sigma e_{\sigma_1}\otimes e_{\sigma_2} \otimes e_{\sigma_3, \sigma_4 ,\sigma_5} \otimes e_{\sigma_6, \sigma_7 ,\sigma_8}. 
\]
Aplying multiplication maps $V^\vee \otimes  V^\vee_{\lambda_2} \to V^\vee_{\lambda_3}$ and $(V^\vee_{\lambda_3})^{\otimes 4} \to \extp^4 V^\vee_{\lambda_3}$ we get $V_{2\lambda_2 + \lambda_8}^\vee \subset \extp^4 V^\vee_{\lambda_3}$ determined by 
\[
w_{2,2,8} = \sum_{\sigma \in \mathcal{S}_8} (-1)^\sigma e_{1,2,\sigma_1}\wedge e_{1,2,\sigma_2}\wedge e_{\sigma_3, \sigma_4 ,\sigma_5} \wedge e_{\sigma_6, \sigma_7 ,\sigma_8}.
\]

Let us show that $\xi \circ \Psi^\vee_{\PP(V_{\lambda_3})} (w_{2,2,8}) \neq 0$. Firstly by Lemma \ref{lem:xi} one computes $\Psi^\vee_{\PP(V_{\lambda_3})} (w_{2,2,8}) =$
$$
\sum_{\sigma \in \mathcal{S}_8} (-1)^\sigma  [e_{1,2,\sigma_1}\wedge e_{1,2,\sigma_2}\cdot e_{\sigma_3, \sigma_4 ,\sigma_5} \wedge e_{\sigma_6, \sigma_7 ,\sigma_8}-2 e_{1,2,\sigma_1}\wedge e_{\sigma_3, \sigma_4 ,\sigma_5}\cdot e_{1,2,\sigma_2} \wedge e_{\sigma_6, \sigma_7 ,\sigma_8}]
$$
and
$$
\xi \circ \Psi^\vee_{\PP(V_{\lambda_3})} (w_{2,2,8}) =  \sum_{\sigma \in \mathcal{S}_8} (-1)^\sigma  [e_A+e_B+e_C],$$
where $e_A=e_{1,2,\sigma_1}\wedge e_{1,2,\sigma_2}\cdot e_{\sigma_3, \sigma_4 ,\sigma_5,\sigma_6, \sigma_7 ,\sigma_8}$, $e_B=-2 e_{1,2,\sigma_1}\wedge e_{\sigma_3, \sigma_4 ,\sigma_5}\cdot e_{1,2,\sigma_2,\sigma_6, \sigma_7 ,\sigma_8}$, $e_C=-2 e_{1,2,\sigma_2} \wedge e_{\sigma_6, \sigma_7 ,\sigma_8} \cdot  e_{1,2,\sigma_1,\sigma_3, \sigma_4 ,\sigma_5} $.
In order to show that $\xi \circ \Psi^\vee_{\PP(V_{\lambda_3})} (w_{2,2,8})  \neq 0$, we want to show that the coefficient of $\xi \circ \Psi^\vee_{\PP(V_{\lambda_3})} (w_{2,2,8})$ corresponding to the element $e_{1,2,3}\wedge e_{1,2,4}\cdot e_{1,2,5,6,7,8}$ with respect to the basis of Remark \ref{rem_basis} is nonzero. Proceeding similarly to the previous proof, one can show that the coefficients of $\sum_{\sigma \in \mathcal{S}_8} (-1)^\sigma e_A$, $\sum_{\sigma \in \mathcal{S}_8} (-1)^\sigma e_B$, $\sum_{\sigma \in \mathcal{S}_8} (-1)^\sigma e_C$ are respectively $2\cdot 6!$, $4\cdot 3!4!$, $4\cdot 3!4!$, and the final coefficient is equal to $2^53^4\neq 0$.
\end{proof}

\begin{Lemma}\label{lem:compl2l3l7}
The equality \eqref{eq:nullintermu} holds for $\mu = \lambda_2+\lambda_3+\lambda_7$.
\end{Lemma}

\begin{proof}
Consider $V_{\lambda_2+ \lambda_3 + \lambda_7}^\vee \subset V_{\lambda_2}^\vee\otimes V_{\lambda_3}^\vee\otimes V_{\lambda_7}^\vee$ determined by $e_{1,2}\otimes e_{1,2,3} \otimes e_{1, \dots,7}$. Then consider the diagonal map
\[
V_{\lambda_7}^\vee \hookrightarrow V^\vee \otimes V_{\lambda_3}^\vee \otimes V_{\lambda_3}^\vee : \quad e_{1, \dots,7} \longmapsto \sum_{\sigma \in \mathcal{S}_7} (-1)^\sigma e_{\sigma_1}\otimes  e_{\sigma_2,\sigma_3, \sigma_4 } \otimes e_{\sigma_5,\sigma_6, \sigma_7}. 
\]
Aplying multiplication maps $V^\vee \otimes  V^\vee_{\lambda_2} \to V^\vee_{\lambda_3}$ and $(V^\vee_{\lambda_3})^{\otimes 4} \to \extp^4 V^\vee_{\lambda_3}$ we get $V_{\lambda_2 + \lambda_3 + \lambda_7}^\vee \subset \extp^4 V^\vee_{\lambda_3}$ determined by
\[
w_{2,3,7} = e_{1,2,3} \wedge \sum_{\sigma \in \mathcal{S}_7} (-1)^\sigma  e_{1,2,\sigma_1}\wedge e_{\sigma_2,\sigma_3, \sigma_4 } \wedge e_{\sigma_5,\sigma_6, \sigma_7}.
\]

Let us show that $\xi \circ \Psi^\vee_{\PP(V_{\lambda_3})} (w_{2,3,7}) \neq 0$. A direct computation shows that $\Psi^\vee_{\PP(V_{\lambda_3})} (w_{2,3,7}) =$
$$
\sum_{\sigma \in \mathcal{S}_7} (-1)^\sigma  [e_{1,2,3} \wedge e_{1,2,\sigma_1}\cdot e_{\sigma_2,\sigma_3, \sigma_4 } \wedge e_{\sigma_5,\sigma_6, \sigma_7} -2 e_{1,2,3} \wedge e_{\sigma_2,\sigma_3, \sigma_4 }\cdot e_{1,2,\sigma_1} \wedge e_{\sigma_5,\sigma_6, \sigma_7}]
$$
and
$$
\xi \circ \Psi^\vee_{\PP(V_{\lambda_3})} (w_{2,3,7}) =  \sum_{\sigma \in \mathcal{S}_7} (-1)^\sigma  [e_A+e_B+e_C],$$
where $e_A=e_{1,2,3} \wedge e_{1,2,\sigma_1}\cdot e_{\sigma_2,\sigma_3, \sigma_4 ,\sigma_5,\sigma_6, \sigma_7}$, $e_B=-2 e_{1,2,3} \wedge e_{\sigma_2,\sigma_3, \sigma_4 }\cdot e_{1,2,\sigma_1,\sigma_5,\sigma_6, \sigma_7}$, $e_C=-2 e_{1,2,\sigma_1} \wedge e_{\sigma_5,\sigma_6, \sigma_7}\cdot e_{1,2,3,\sigma_2,\sigma_3, \sigma_4 } $.
In order to show that $\xi \circ \Psi^\vee_{\PP(V_{\lambda_3})} (w_{2,3,7})  \neq 0$, we want to show that the coefficient of $\xi \circ \Psi^\vee_{\PP(V_{\lambda_3})} (w_{2,3,7})$ corresponding to the element $e_{1,2,3}\wedge e_{1,2,7}\cdot e_{1,2,3,4,5,6}$ with respect to the basis of Remark \ref{rem_basis} is nonzero. As in the previous proofs, one can show that the coefficients of $\sum_{\sigma \in \mathcal{S}_7} (-1)^\sigma e_A$, $\sum_{\sigma \in \mathcal{S}_7} (-1)^\sigma e_B$, $\sum_{\sigma \in \mathcal{S}_7} (-1)^\sigma e_C$ are respectively $ 6!$, $2\cdot 3!4!$, $-3!4!$, and the final coefficient is equal to $2^53^3\neq 0$.
\end{proof}

\begin{Lemma}\label{lem:compl1l4l7}
The equality \eqref{eq:nullintermu} holds for $\mu = \lambda_1+\lambda_4+\lambda_7$.
\end{Lemma}

\begin{proof}
Consider $V_{\lambda_1+ \lambda_4 + \lambda_7}^\vee \subset V_{\lambda_1}^\vee\otimes V_{\lambda_4}^\vee\otimes V_{\lambda_7}^\vee$ determined by $e_{1}\otimes e_{1,2,3,4} \otimes e_{1, \dots,7}$. Then consider the diagonal maps
\begin{align*}
    V_{\lambda_7}^\vee \hookrightarrow V^\vee \otimes V_{\lambda_3}^\vee \otimes V_{\lambda_3}^\vee & : \quad e_{1, \dots,7} \longmapsto \sum_{\sigma \in \mathcal{S}_7} (-1)^\sigma e_{\sigma_1}\otimes  e_{\sigma_2,\sigma_3, \sigma_4 } \otimes e_{\sigma_5,\sigma_6, \sigma_7}; \\
    V_{\lambda_4}^\vee \hookrightarrow V^\vee \otimes V_{\lambda_3}^\vee & : \quad e_{1,2,3,4}\longmapsto \sum_{\sigma \in \mathcal{S}_4} (-1)^\sigma e_{\sigma_1}\otimes  e_{\sigma_2,\sigma_3, \sigma_4 }.
\end{align*}
Aplying multiplication maps $V^\vee \otimes  V^\vee_{\lambda_2} \to V^\vee_{\lambda_3}$ and $(V^\vee_{\lambda_3})^{\otimes 4} \to \extp^4 V^\vee_{\lambda_3}$ we get $V_{\lambda_1 + \lambda_4 + \lambda_7}^\vee \subset \extp^4 V^\vee_{\lambda_3}$ determined by
\[
w_{1,4,7} = \sum_{\sigma \in \mathcal{S}_4} \sum_{\tau \in \mathcal{S}_7} (-1)^\sigma(-1)^\tau e_{1,\sigma_1, \tau_1} \wedge  e_{\sigma_2,\sigma_3, \sigma_4 } \wedge e_{\tau_2,\tau_3,\tau_4} \wedge e_{\tau_5,\tau_6,\tau_7}.
\]

Let us show that $\xi \circ \Psi^\vee_{\PP(V_{\lambda_3})} (w_{1,4,7}) \neq 0$. Firstly notice that 
\[
\begin{split}
\Psi^\vee_{\PP(V_{\lambda_3})} (w_{1,4,7}) = \sum_{\sigma \in \mathcal{S}_4} \sum_{\tau \in \mathcal{S}_7} (-1)^\sigma(-1)^\tau [e_{1,\sigma_1, \tau_1} \wedge  e_{\sigma_2,\sigma_3, \sigma_4 } \cdot e_{\tau_2,\tau_3,\tau_4} \wedge e_{\tau_5,\tau_6,\tau_7}  \\ - 2e_{1,\sigma_1, \tau_1} \wedge e_{\tau_2,\tau_3,\tau_4} \cdot e_{\sigma_2,\sigma_3, \sigma_4 }  \wedge e_{\tau_5,\tau_6,\tau_7}]
\end{split}
\]
and
$$
\xi \circ \Psi^\vee_{\PP(V_{\lambda_3})} (w_{1,4,7}) =  \sum_{\sigma \in \mathcal{S}_4} \sum_{\tau \in \mathcal{S}_7} (-1)^\sigma(-1)^\tau [e_A+e_B+e_C],$$
where $e_A=e_{1,\sigma_1, \tau_1} \wedge  e_{\sigma_2,\sigma_3, \sigma_4 } \cdot e_{\tau_2,\tau_3,\tau_4,\tau_5,\tau_6,\tau_7}$, $e_B=- 2e_{1,\sigma_1, \tau_1} \wedge e_{\tau_2,\tau_3,\tau_4} \cdot e_{\sigma_2,\sigma_3, \sigma_4,\tau_5,\tau_6,\tau_7}$, $e_C=- 2 e_{\sigma_2,\sigma_3, \sigma_4 }  \wedge e_{\tau_5,\tau_6,\tau_7} \cdot e_{1,\sigma_1, \tau_1,\tau_2,\tau_3,\tau_4}$.
In order to show that $\xi \circ \Psi^\vee_{\PP(V_{\lambda_3})} (w_{1,4,7})  \neq 0$, we want to show that the coefficient of $\xi \circ \Psi^\vee_{\PP(V_{\lambda_3})} (w_{1,4,7})$ corresponding to the element $e_{1,2,3}\wedge e_{1,2,4}\cdot e_{1,3,4,5,6,7}$ with respect to the basis of Remark \ref{rem_basis} is nonzero. Similarly to the previous proofs, one can compute that the coefficient of $\sum_{\sigma \in \mathcal{S}_4} \sum_{\tau \in \mathcal{S}_7} (-1)^\sigma(-1)^\tau e_A$ is equal to $12\cdot 6!$, the coefficient of $\sum_{\sigma \in \mathcal{S}_4} \sum_{\tau \in \mathcal{S}_7} (-1)^\sigma(-1)^\tau e_B$ is equal to $4\cdot (3!)^3$ and the coefficient of $\sum_{\sigma \in \mathcal{S}_4} \sum_{\tau \in \mathcal{S}_7} (-1)^\sigma(-1)^\tau e_C$ is equal to $4\cdot (3!)^2 \cdot 4!$, thus giving a total coefficient of $2^53^45\neq 0$.
\end{proof}


\def\cprime{$'$}

\end{document}